\documentclass[11pt]{article}

\usepackage[a4paper,headinclude=false,footinclude=false,DIV=10]{typearea}
\usepackage{setspace}

\usepackage{multicol}

\usepackage{amsmath,amsthm,amssymb}
\usepackage{stmaryrd}

\usepackage{epsfig,graphicx}
\usepackage[colorlinks=false]{hyperref}
\usepackage[english]{babel}

\usepackage[utf8]{inputenc}
\usepackage[T1]{fontenc}

\usepackage{fourier} 
\usepackage[scaled=0.875]{helvet} 

\newtheoremstyle{mystyle}{}{}{\rmfamily}%
{}{\normalfont\bfseries}{ }{ }{}

\newtheorem{theorem}{Theorem}[section]
\newtheorem{proposition}[theorem]{Proposition}

\newtheorem{corollary}[theorem]{Corollary}
\newtheorem{definition}[theorem]{Definition}

\newtheorem{introtheorem}{Theorem}

\theoremstyle{mystyle}
\newtheorem{remark}[theorem]{Remark}

\newcommand{\R}{\mathbb{R}}

\newcommand{\N}{\mathbb{N}}

\newcommand{\G}{\mathbb{G}}
\newcommand{\X}{\mathbb{X}}

\newcommand{\dd}{\mathop{}\!\mathrm{d}}

\newcommand{\cA}{\mathcal{A}}

\newcommand{\cL}{\mathcal{L}}

\DeclareMathOperator{\id}{id}

\usepackage{enumitem}
\setlist[enumerate]{itemsep=0cm,parsep=\smallskipamount}

\usepackage{xcolor}
\usepackage{tikz}

\usepackage[
style=trad-abbrv,
isbn=false,
doi=false,
url=false,
backend=biber]{biblatex}

\IfFileExists{/home/manuels/Math/Math-In-The-Cloud/Papers/cloud_library.bib}{
\addbibresource{~/Math/Math-In-The-Cloud/Papers/cloud_library.bib}
}{
\addbibresource{~/Math/Papers/cloud_library.bib}
}

\begin{document}
\selectlanguage{english}
\vspace*{1cm}

\begin{center}
{\huge \bfseries On fluctuations of the drift in hyperbolic groups}\\[.5cm]
Gil José Astudillo Hernandez$^\mathrm{a}$, Manuel Stadlbauer$^\mathrm{a}$

\bigskip
{\scriptsize
$^\mathrm{a}$ Departamento de Matemática, Universidade Federal do Rio de Janeiro,\\
Ilha do Fundão,  21941-909 Rio de Janeiro (RJ), Brazil\\[-.2cm]
}
\bigskip \bigskip
{\small \today}
\end{center}

\begin{abstract}
Let $\Gamma$ be a convex-cocompact group of isometries of a CAT($-1$) space $X$
and let $Y \to X/\Gamma$  be a Galois cover with
a word-hyperbolic group of deck
transformations. We show that, for almost every geodesic $\xi$ with respect to the Bowen-Margulis-Sullivan measure on the lift of $X/\Gamma$,   there exist $\mathfrak{m}, \sigma > 0$ and a standard Brownian motion $B_s$ such
that, for any $\lambda > 1/4$,
\[
    d(p(g_s (\xi)), \mathbf{o}) = \mathfrak{m}s + \sigma B_s + o(s^\lambda),
\]
where $g_s$ denotes the geodesic flow acting on the geodesics of $Y$ and $p(g_s)$  the canonical projection to $Y$. The result is a consequence of an almost sure invariance principle for random walks on hyperbolic groups with dependent increments. Its proof makes use of a new Ruelle operator theorem for skew products and Martin boundary techniques for random walks with dependent increments.\\


\end{abstract}

\section{Introduction}
In this article, we prove a Ruelle operator theorem for skew products under a mild hypothesis on contraction in average, discuss its consequences with respect to approximations by Brownian motion and synchronization, and then show how to apply these results to random walks on hyperbolic groups with not necessarily independent increments.

The main application and motivation for this sequence of results is a precise characterization of the asymptotic behavior of the distance from the starting point along the geodesic flow in the following class of geometrically infinite CAT($-1$) manifolds. Recall that  CAT($-1$) space $X$ is a uniquely geodesic metric space where every geodesic triangle is thinner than its comparison triangle in the hyperbolic plane of constant curvature $-1$.
Let $\mathcal{G}$ refer to the space of complete geodesics in $X$, and assume that $\Gamma$ is a  discrete group of isometries such that a $\mathcal{G}_\Gamma/\Gamma$ is compact, with $\mathcal{G}_\Gamma$ referring to those geodesics with endpoints in the limit set of $\Gamma$. Furthermore, we assume that there is an intermediate, locally isometric cover $X \to Y \to X/\Gamma$ such that the fundamental group $N$ of $Y$ is normal in $\Gamma$  (see Figure \ref{fig:intermediate-cover} for an example in which  $\Gamma/N$ is the free group on two generators). In other words, $\Gamma$ is convex-cocompact and $Y$ is a Galois cover (or regular cover) of $X/\Gamma$. In this situation, it is well known that the Bowen-Margulis-Sullivan measure with respect to $\Gamma$ is supported on $\mathcal{G}_\Gamma/\Gamma$ and that the group of deck transformations of $Y \cong X/N \to X/\Gamma$ is isomorphic to $\Gamma/N$.
\begin{figure}[ht]
\centering\def\svgwidth{0.91\textwidth}
\begingroup%
  \makeatletter%
  \providecommand\color[2][]{%
    \errmessage{(Inkscape) Color is used for the text in Inkscape, but the package 'color.sty' is not loaded}%
    \renewcommand\color[2][]{}%
  }%
  \providecommand\transparent[1]{%
    \errmessage{(Inkscape) Transparency is used (non-zero) for the text in Inkscape, but the package 'transparent.sty' is not loaded}%
    \renewcommand\transparent[1]{}%
  }%
  \providecommand\rotatebox[2]{#2}%
  \newcommand*\fsize{\dimexpr\f@size pt\relax}%
  \newcommand*\lineheight[1]{\fontsize{\fsize}{#1\fsize}\selectfont}%
  \ifx\svgwidth\undefined%
    \setlength{\unitlength}{1164.85459215bp}%
    \ifx\svgscale\undefined%
      \relax%
    \else%
      \setlength{\unitlength}{\unitlength * \real{\svgscale}}%
    \fi%
  \else%
    \setlength{\unitlength}{\svgwidth}%
  \fi%
  \global\let\svgwidth\undefined%
  \global\let\svgscale\undefined%
  \makeatother%
  \begin{picture}(1,0.35915831)%
    \lineheight{1}%
    \setlength\tabcolsep{0pt}%
    \put(0,0){\includegraphics[width=\unitlength,page=1]{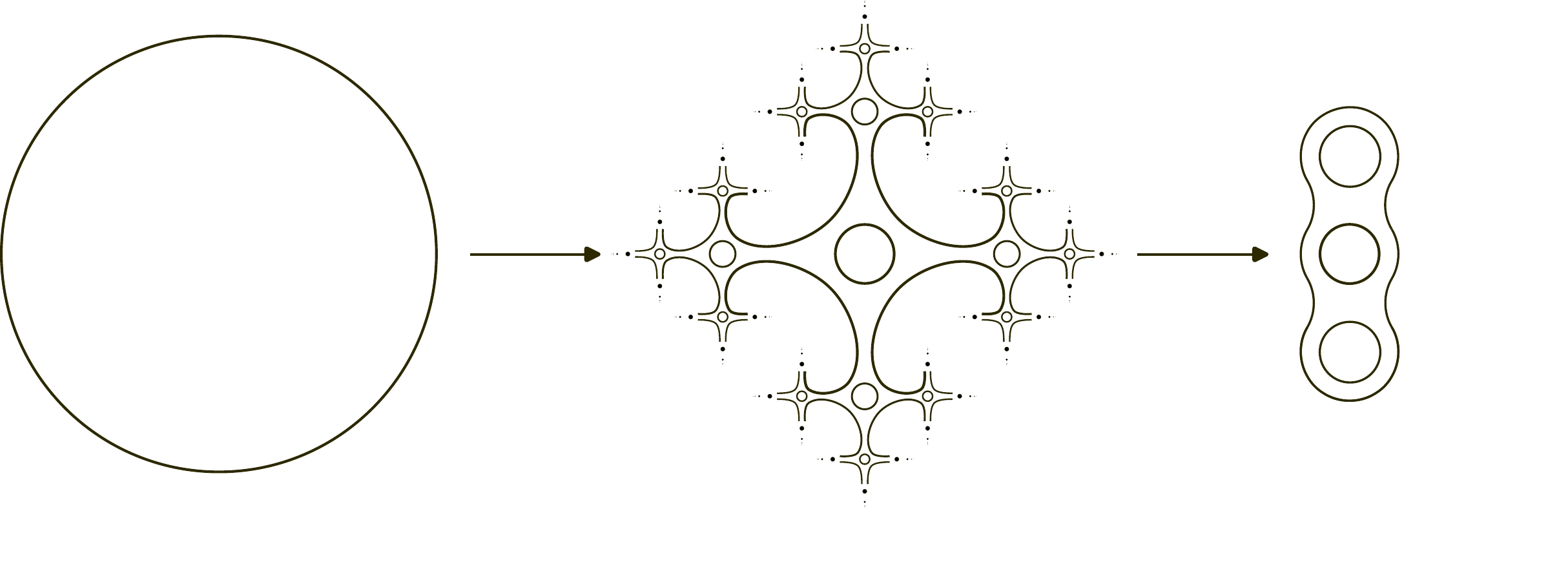}}%
    \put(0.13178566,0.00091719){\color[rgb]{0.17254902,0.16078431,0}\makebox(0,0)[lt]{\lineheight{0.72000158}\smash{\begin{tabular}[t]{l}$X$\end{tabular}}}}%
    \put(0.53097729,0.00091719){\color[rgb]{0.17254902,0.16078431,0}\makebox(0,0)[lt]{\lineheight{0.72000158}\smash{\begin{tabular}[t]{l}$X/N$\end{tabular}}}}%
    \put(0.84002826,0.00091719){\color[rgb]{0.17254902,0.16078431,0}\makebox(0,0)[lt]{\lineheight{0.72000158}\smash{\begin{tabular}[t]{l}$X/\Gamma$\end{tabular}}}}%
  \end{picture}%
\endgroup%

\caption{Intermediate cover}
\label{fig:intermediate-cover}
\end{figure}
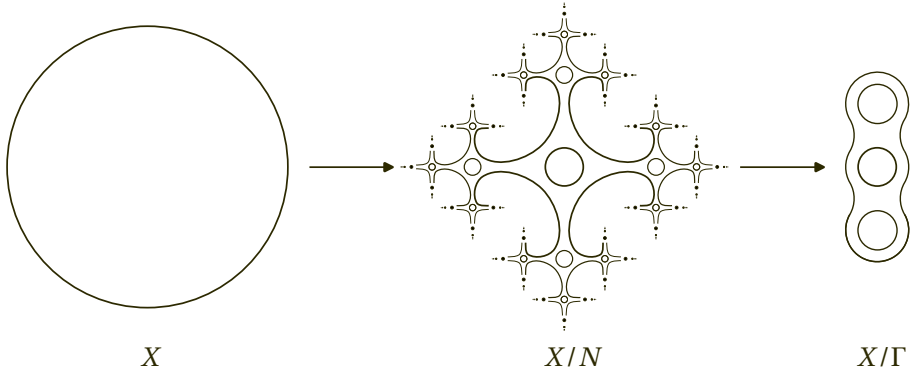

If, in addition, $\Gamma/N$ is a word-hyperbolic group, it follows that the natural projection $ p: \mathcal{G}_\Gamma/N \to Y$ is roughly isometric to $\Gamma/N$, which implies that the Gromov boundary $\partial(\Gamma/N)$ of  $\Gamma/N$ and the one of $p (\mathcal{G}_\Gamma/N)$   are homeomorphic.

Furthermore, recall that the geodesic flow on $\mathcal{G}$ is defined by $g_s(\gamma(t)) := \gamma(t+ s)$, for $s,t \in \R$ and $(\gamma: \R \to X) \in \mathcal{G}$, and note that $(g_s)$ projects onto the geodesic flows on $\mathcal{G}/N$  and $\mathcal{G}/\Gamma$, respectively. To obtain a characterization of the asymptotic behavior, it remains to define, for $x,y \in Y$
\[
d_Y^\ast(x,y) := - \log \sum_{(\ast)} e^{- \delta_\Gamma \ell(\gamma)},
\]
where $(\ast)$ stands for the summation over all local geodesic arcs from $x$ to $y$, $\ell$ for the length of the arc with respect to the metric $d_Y$ on $Y$, and $\delta_\Gamma$ for the exponent of convergence of $\Gamma$. At this point, it is worth noting that $d_Y^\ast$ is the Poincaré series of $N$ with respect to the exponent of convergence of $\Gamma$ and that, as shown in Theorem \ref{theo:geometric-application}, there exist $c_1,c_2 > 0$ such that
\[  c_1 d_Y  - c_2 \leq  d_Y^\ast \leq \delta_\Gamma d_Y. \]
\begin{introtheorem} \label{mainthm:vasip-for-the-flow}
Assume that $\Gamma/N$ is a non-elementary word-hyperbolic group, and that $\mathbf{o} \in Y$. Then, there exist a measurable, non-continuous map $\pi:  \mathcal{G}_\Gamma/N  \to \partial(\Gamma/N)$ and $\mathfrak{m},\sigma > 0$ such that, for almost every $\xi \in \mathcal{G}_\Gamma/N$ with respect to the lift of the  Bowen-Margulis-Sullivan measure,
\begin{enumerate}
 \item $d_Y(p(g_s (\xi)),\mathfrak{g}_\xi([0,\infty)) ) = O(\log s)$ as $s \to \infty$,  for any geodesic half ray $\mathfrak{g}_\xi$ from  $\mathbf{o}$ to $\pi(\xi)$,
 \item after eventually extending the measure space, there exists a standard Brownian motion $B_s$ such that,  for any $\lambda > 1/4$,
 \[d_Y^\ast(p(g_s (\xi)),\mathbf{o}) = \mathfrak{m} s + \sigma B_s + o (s^{\lambda}).\]
 \item If $p(\mathcal{G}_\Gamma/N)$ is strongly hyperbolic, then one may replace $d_Y^\ast$ with $d_Y$ in the above approximation by a Brownian motion.
\end{enumerate}
\end{introtheorem}
At this point, we would like to include some remarks and observations concerning the generality of the CAT($-1$) setting and  the relevance of the approximation by Brownian motion, as well as related results for simple random walks on word-hyperbolic groups.

With respect to the generality, note that, if $X/\Gamma$ is a compact manifold with variable, strictly negative curvature, then the exponent of convergence $\delta_\Gamma$ is equal to the conformal dimension of $\partial X$, and $\mathcal{G}$, $\mathcal{G}/N$ and $\mathcal{G}/\Gamma$ are the sphere bundles over $X$, $X/N$ and $X/\Gamma$, respectively.

The second statement of Theorem \ref{mainthm:vasip-for-the-flow} is an almost sure invariance principle and is of particular interest since it allows transferring results for the Brownian motion to this geometric setting. The standard law of the iterated logarithm implies that
\begin{equation} \label{eq:second-order-approx}
\left| d_Y^\ast(p(g_s \xi),\mathbf{o})  -
\mathfrak{m} s\right| \leq  \sqrt{2 \sigma^2 s \log \log s} \,  \left(1 + {o}(1)\right)
\end{equation}
and that the estimate is sharp. It is worth noting that \eqref{eq:second-order-approx} represents a second order approximation of the fluctuation, and that, as $g_s(\xi)$ always stays within an error of order $\log s$ to the half ray $\gamma_\xi([0,\infty))$ by the
first statement of Theorem \ref{mainthm:vasip-for-the-flow}, these fluctuations correspond to moving forwards and backwards along the geodesic half ray.
However, more advanced probabilistic laws provide a refined description of this almost sure error term. For example, Strassen's functional law of the iterated logarithm implies that, for any absolutely continuous function $g: [0,1] \to \R$ with $g(0)=0$ and
$\| g' \|_2 \leq 1$ and almost every $\xi$, there exists a sequence $t_n \to \infty$ such that, for $\epsilon_n(\xi,s)$ given by
\[
d_Y^\ast(p(g_{st_n} \xi),\mathbf{o}) = \mathfrak{m} s t_n + g(s) \sqrt{2 \sigma^2 t_n \log \log t_n} \,  \left(1 + \epsilon_n(\xi,s)\right),
\]
one has that $\lim_{n \to \infty} \sup \{ |\epsilon_n(\xi,s)  | : s \in [0,1] \} =0$. This obviously excludes the possibility of a third order approximation of the drift.

On the other hand, as it is well known, the weak invariance principle and the central limit theorem are consequences of the almost sure invariance principle. However, to do so, one has to choose a fundamental domain  $P \subset Y$ for the action of the deck transformations on $Y$ and normalize the Bowen-Margulis-Sullivan measure on $X/\Gamma$, which appears to be artificial in this context. Moreover, due to the nature of weak convergence, one only has access to the asymptotic distribution of $(d_Y^\ast(p(g_s \xi),\mathbf{o}) -  \mathfrak{m}s)/\sqrt{s}$, which does provides us with a second order approximation as in \eqref{eq:second-order-approx}.

There are several parallel results for random walks with independent increments on word-hyperbolic groups. For example, the first assertion, with an error of order \(o(n)\), is due to Kaimanovich (\cite{Kaimanovich--The-Poisson-Formula-For--AM22000}) and was subsequently sharpened to \(O(\log n)\) by Maher and Tiozzo for weakly hyperbolic groups (\cite{Maher-Tiozzo--Random-Walks-On-Weakly--JRAM-2018}). Moreover, a central limit theorem for the displacement, as in the second assertion, was obtained by Ledrappier for free groups (\cite{Ledrappier--Some-Asymptotic-Properties-Of-Random-Walks-On-Free--2001}), by Björklund (\cite{Bjorklund--Central-Limit-Theorems-For--JTP-2010}) for word-hyperbolic groups and, in greater generality, by Benoist and Quint (\cite{Benoist-Quint--Random-Walks-On-Reductive--2016,Benoist-Quint--Central-Limit-Theorem-On-Hyperbolic-Groups--IRANSM2016}), Gouëzel (\cite{Gouezel--Analyticity-Of-The-Entropy--DA2017}) and Mathieu \& Sisto (\cite{Mathieu-Sisto--Deviation-Inequalities-For-Random-Walks--DMJ2020}).
A central limit theorem for the displacement and the drift were obtained
by Gekhtman, Taylor \& Tiozzo in \cite{Gekhtman-Taylor-Tiozzo--Central-Limit-Theorems-For-Counting-Measures-In-Coarse--CM2022} for coarse hyperbolic spaces. Furthermore, Choi obtained a refined tracking result by geodesics as well as a central limit theorem and law of the iterated logarithm for the displacement in \cite{Choi--Central-Limit-Theorem-And-Geodesic-Tracking-On-Hyperbolic--AM2023}.

The large deviation principles for the drift by Boulanger, Mathieu, Sert \& Sisto in   \cite{Boulanger-Mathieu-Sert-Sisto--Large-Deviations-For-Random-Walks-On-Gromov-Hyperbolic--ASENS42023} and by Aoun, Mathieu \& Sert in \cite{Aoun-Mathieu-Sert--Random-Walks-On-Hyperbolic-Spaces-Second-Order-Expansion--JEPM2023} (see also \cite{Gouezel--Exponential-Bounds-For-Random-Walks-On-Hyperbolic-Spaces--TJM2022}) as well as the concentration inequalities by Aoun \& Sert in \cite{Aoun-Sert--Random-Walks-On-Hyperbolic-Spaces-Concentration-Inequalities-And--PTRF2022} and the central limit theorem for equivariant cocycles with values in $\R^d$ by Bénard in \cite{Benard--Winding-Of-Geodesic-Rays-Chosen-By-A-Harmonic--MA2024} provide further evidence of the probabilistic character of the drift and the displacement. Furthermore, several authors have obtained central limit theorems for the average displacement of group elements in spheres with respect to a second metric on \(G\) (see \cite{Calegari-Fujiwara--Combable-Functions-Quasimorphisms-And--ETDS2010,Cantrell--Typical-Behaviour-Along-Geodesic-Rays-In-Hyperbolic-Groups--MZ2021,Cantrell--Statistical-Limit-Laws-For-Hyperbolic-Groups--TAMS2021,Cantrell--Mixing-Of-The-Mineyev--MA-2025,Cantrell-Pollicott--Central-Limit-Theorems-For-Green-Metrics-On-Hyperbolic--AHL2026}, and references therein).

However, to the best of the authors' knowledge, no almost sure invariance principle is known for the displacement of the random walk, nor are there analogous limit laws for the geodesic flow as in Theorem \ref{mainthm:vasip-for-the-flow}.

\bigskip

We now pass to the structure of this paper, which is divided into two principal parts. In the first part, we prove a vector-valued almost sure invariance principle  for suspension semi-flows over a class of non-expanding skew products, whereas the second establishes an almost sure invariance principle for the drift of random walks on hyperbolic groups with dependent increments. Theorem \ref{mainthm:vasip-for-the-flow} then follows as a special case of this result.

In Section \ref{sec:ruelle-for-skew}, we introduce the class of skew products whose fiber maps are homeomorphisms of a compact metric space. In particular, this implies that expansion and contraction of distances in the fibers always coexist and that one should expect phenomena known from random walks on $\mathbb{S}^1$  (cf. \cite{Malicet--Random-Walks-On-Rm--CMP-2017}) or random walks on general metric spaces with a local contraction property (cf. \cite{Gelfert-Salcedo--Synchronization-Rates-And-Limit--MZ-2024}).
However, to access probabilistic limit theorems through the spectral method (cf. \cite{Gouezel--Almost-Sure-Invariance-Principle--AP2010}), it is necessary to establish a Ruelle operator theorem, which is the content of this section.

Now assume that $(\Sigma,d_\Sigma)$ and  $(X,d_X)$ are compact, metric spaces, that $\theta : \Sigma \to \Sigma$ is Ruelle expanding (see Def. \ref{def:Ruelle-expanding}), and that $\varphi: \Sigma \to \R$ is a Hölder continuous function such that $\sum_{\theta\xi^\ast = \xi} \exp(\varphi(\xi^\ast)) =1$ for all $\xi \in \Sigma$.
A map $\kappa: \Sigma \to \mathrm{Homeo}(X)$ from $\Sigma$ to the space of homeomorphisms of $X$ then gives rise to  a skew product $T: \Sigma \times X \to  \Sigma \times X$ and its associated transfer operator $\cL$, acting on positive functions by
\[
T(\xi,x) := (\theta(\xi),\kappa(\xi)(x)), \quad  \cL(f)(\xi,x) :=   \sum_{T (\xi^\ast,x^\ast) = (\xi,x)} e^{\varphi(\xi^\ast)} f(\xi^\ast,x^\ast).
\]
It is now crucial to observe that, as $\theta$ is Ruelle expanding, there is some $a > 0$ such that for any $n \in \N$ and $\xi \in \Sigma$, there is a bijection from the set $\theta^{-n}(\{\xi\})$ to the set of inverse branches of $T^n$, defined on the ball of radius $a$ with center $\xi$. We then say that $(T,\varphi)$ is locally eventually backward contracting in average if, for some $a> 0$ sufficiently small and $n$ sufficiently large, and some metric $d$ on $\Sigma \times X$,
\begin{align} \label{eq:bca-intro}
\sup_{0< d(\zeta_1,\zeta_2) < a} \left(  \frac{1}{d(\zeta_1,\zeta_2)} \sum_{(\ast)}  e^{\varphi(\xi^\ast)} d(T^{-n}_{\xi^\ast}(\zeta_1),T^{-n}_{\xi^\ast}(\zeta_2)) \right) < 1,     \end{align}
where $(\ast)$ refers to the summation over $\theta^{-n}(\{\xi\})$, for $(\xi,x) = \zeta_1$, for some $x \in X$ and $T^{-n}_{\xi^\ast}$ to inverse branch associated to $\xi^\ast \in \theta^{-n}(\{\xi\})$. Moreover, if the additional, mild assumption of communicating preimages holds, we say that $(T,\varphi)$ is backward contracting in average (see Def. \ref{def:LEEA}). The main result in Section \ref{sec:ruelle-for-skew} can be summarized as follows (cf. Theorems \ref{theo:Wasserstein} and \ref{theo:Ruelle}).
\begin{introtheorem}\label{mainthm:bca}
Assume that $(\Sigma,\theta)$ is Ruelle expanding and that $(T,\varphi)$ is backward contracting in average. Then there are $\ell > 0$ and $\tau \in (0,1)$ such that the action of the dual of $\cL^\ell$ on the space of probability measures is a contraction in the Wasserstein distance by a factor of $\tau$. Furthermore, there exists a unique probability measure $\mu$ such that, for the space of Hölder functions on $\Sigma \times X$,
\[ \| \cL^{\ell k} (f)  - \int f \dd \mu \| \ll \tau^k D(f),    \]
 where $D(f)$ refers to the Hölder coefficient of $f$, and $\|f\|$ to the Hölder norm $\|f\|_\infty + D(f)$. If the inverse branches of $T$ are Lipschitz continuous, one may suppose that $\ell =1$ in the above contraction of Hölder functions.
\end{introtheorem}

We emphasize that neither the statement nor the proof of Theorem \ref{mainthm:bca} makes use of partitions. This allows the theorem to be applied in the setting of expanding local diffeomorphisms on manifolds with a differentiable cocycle $\kappa$. Moreover, Theorem \ref{mainthm:bca} yields a contraction result for Hölder functions on $\Sigma \times X$, in contrast to the setting of quenched, annealed, or sectional transfer operators, which typically focus on functions defined on $X$.

Another important feature is that Theorem \ref{mainthm:bca} requires local contraction on average. By contrast, standard notions of contraction on average usually require either that, for some $\delta>0$ (with $\mathrm{Lip}$ denoting the Lipschitz constant),
\[  \sum_{\theta \xi^\ast = \xi}  e^{\varphi(\xi^\ast)} \log \mathrm{Lip}(T^{-1}_{\xi^\ast}) < -\delta,   \]
as in \cite{Diaconis-Freedman--Iterated-Random-Functions--SIAMR1999,Jordan-Pollicott--The-Hausdorff-Dimension-Of--DCDS2008,Alves-Bahsoun--Decay-Of-Correlations-For-Partially-Hyperbolic-Skew-Products--2026}, or  that \eqref{eq:bca-intro} holds globally, as in \cite{Werner--Contractive-Markov-Systems--JLMS22005,Kloeckner--Extensions-With-Shrinking-Fibers--ETDS2020}. For a detailed overview of these conditions, we refer to \cite{Gelfert-Salcedo--Contracting-On-Average-Iterated-Function-Systems-By-Metric--N2023} and the references therein. However, at this point, we would like to note that in our setting $\mathrm{Lip}(T^{-1}_{\xi^\ast}) > 1$ which implies that the first class of conditions is never satisfied.

In Section \ref{sec:consequences-of-Ruelle}, we then discuss several  consequences of Theorem \ref{mainthm:bca}. Using Gouëzel's results in \cite{Gouezel--Almost-Sure-Invariance-Principle--AP2010}, we establish an almost sure invariance principle for Lipschitz continuous, $\mathbb{R}^d$-valued functions (see Theorem \ref{theo:vasip}). The invariance principle then allows us to extend this result to suspension semi-flows (see Theorem \ref{theo:vasip-flow}). In addition, our approach yields a slightly sharper error term than those obtained in \cite{Denker-Philipp--Approximation-By-Brownian-Motion--ETDS1984,Rudolph--Asymptotically-Brownian-Skew-Products-Give-Non-Loosely-Bernoulli--IM1988,Melbourne-Nicol--A-Vector-valued-Almost-Sure--AP2009}. Thereafter, we  establish a synchronization result (see Theorem \ref{theo:synchronization}), based on the Wasserstein-metric approach employed in the proof of Theorem \ref{mainthm:bca}.

In Section \ref{sec:random-walks}, we then pass to the analysis of skew products of the form
\begin{equation} \nonumber
 \label{eq:random-walk-with-stationary-increments-intro}
S: \Sigma \times G \to  \Sigma \times G, (\xi,g) \mapsto (\theta\xi, g \gamma(\xi) ),
\end{equation}
where  $(\Sigma,\theta)$ is a topologically transitive subshift of finite type, $G$ is a word-hyperbolic group and $\gamma: \Sigma \to G$ is locally constant. By choosing a Hölder continuous potential $\varphi: \Sigma \to \R$ and its associated equilibrium state $\nu$, the evolution in the second coordinate models a random walk on $G$ with dependent increments. Using the Martin boundary techniques developed in \cite{Bispo-Stadlbauer--The-Martin-Boundary-Of--IJM2023} (cf. \cite{Gouezel--Local-Limit-Theorem-For--JAMS2014} for independent increments) for $(S,\nu)$, allows one to extend results
by Furstenberg, Kaimanovich and Maher \& Tiozzo to the setting of dependent increments in  Theorem \ref{theo:positive-drift}, and Proposition \ref{prop:convergence-to-the-boundary}. That is, we show that $\lim_n  d_w( \id, \gamma(\xi)\cdots  \gamma(\theta^{n-1}\xi))/n > 0$ (cf.
\cite{Furstenberg--Noncommuting-Random-Products--TAMS1963,Kaimanovich--The-Poisson-Formula-For--AM22000}), that $\pi(\xi):= \lim  \gamma(\xi)\cdots  \gamma(\theta^{n-1}\xi) \in \partial G$ exists almost surely (cf. \cite{Kaimanovich--The-Poisson-Formula-For--AM22000}) and that the ray to $\pi(\xi)$ is tracked with an error of order $\log n$ (cf. \cite{Maher-Tiozzo--Random-Walks-On-Weakly--JRAM-2018}).
It then follows from positive drift that the skew product
\begin{equation} \nonumber
 \label{eq:random-walk-on-gromov-boundary-intro}
T: \Sigma \times \partial G \to  \Sigma \times \partial G, (\xi,x) \mapsto (\theta\xi,  \gamma(\xi)^{-1}x )
\end{equation}
is backward contracting in average and, in particular, the vector-valued almost sure invariance principle holds for Lipschitz continuous observables on $\Sigma \times \partial G$. As an application, after passing from the word metric $d_w$ to the Green metric $d_{\G}$, one then obtains a precise description of the fluctuation of the drift (Corollary \ref{cor:fluctuation-of-drift}) and a significant refinement of the central limit theorem by Björklund in \cite{Bjorklund--Central-Limit-Theorems-For--JTP-2010}.
\begin{introtheorem} \label{mainthm:drift}
 Assume that $G$ is a word-hyperbolic, non-elementary group, $T$ is topologically transitive and that $\nu$ is an equilibrium state with respect to some Hölder potential on $\Sigma$. Then
 there exist $\mathfrak{m},\sigma>0$ and a standard Brownian motion $B$ such that, for any $\lambda > 1/4$ and $\nu$ almost every $\xi \in \Sigma$,
  \[
  d_\G(\id, \gamma(\xi)\cdots  \gamma(\theta^{n-1}\xi)) = \mathfrak{m} n + \sigma B_n + o(n^\lambda).
  \]
\end{introtheorem}

Section \ref{sec:regular cover} is  devoted to the proof of Theorem \ref{mainthm:vasip-for-the-flow}. By employing the coding of Constantine, Lafont \& Thompson in \cite{Constantine-Lafont-Thompson--Strong-Symbolic-Dynamics-For--JLP-M2020}  to  the geodesic flow on $\mathcal{G}(X/\Gamma)$, one obtains a non-invertible Markov map $(\Sigma,\theta)$, which provides us with a coding of the unstable part, in analogy to the Bowen-Series map for cocompact Fuchsian groups (cf. \cite{Series--Geometrical-Markov-Coding-Of--ETDS1986,Adler-Flatto--Geodesic-Flows-Interval-Maps--BAMSNS1991}). The unstable part of the flow on $\mathcal{G}(X/N)$ is then coded by the skew product $S$ (cf.  \cite{Bispo-Stadlbauer--The-Martin-Boundary-Of--IJM2023}). Theorem \ref{mainthm:vasip-for-the-flow} then follows from Theorems \ref{theo:vasip-flow} and  \ref{mainthm:drift}. As a concluding remark, we would like to point out that the proof of Theorem \ref{mainthm:vasip-for-the-flow} depends on having a vector-valued almost sure invariance principle at hand, and not just weak convergence to a Brownian motion.

\section{A Ruelle theorem for skew products with compact fibers}\label{sec:ruelle-for-skew}

We recall the basic results of the thermodynamic formalism for Ruelle expanding maps with respect to H\"older continuous potential functions. Throughout, we assume that $(\Sigma,d)$ is a compact, metric space and that $\theta: \Sigma \to \Sigma$ is Ruelle expanding as defined below (cf. \cite{Ruelle--The-Thermodynamic-Formalism-For--CMP1989}).

\begin{definition}[Ruelle expanding maps] \label{def:Ruelle-expanding}
A continuous map $\theta: \Sigma \to \Sigma$ is said to be \emph{$(a,\rho)$-Ruelle-expanding}, for some $a>0$ and  $\rho \in (0, 1)$, if for any $\xi, \eta, \tilde{\xi} \in \Sigma$  with $d(\xi, \eta)<a$ and $\theta(\tilde{\xi})=\xi$, there exists a unique $\tilde{\eta}\in \Sigma$  with $\theta(\tilde{\eta})={\eta}$ and
$ d(\tilde{\xi}, \tilde{\eta}) \leq  \rho d(\xi,\eta)$.
\end{definition}

Moreover, recall that $\theta$ is referred to as \emph{topologically transitive} if for any pair of open subsets  $U,V \subset \Sigma$, there exists $n \in \N$ such that $\theta^{-n}(U) \cap V \neq \emptyset$, and that   $\theta$ is referred to as \emph{topologically mixing} if
for  any pair of open subsets $U,V \subset \Sigma$, there exists $n \in \N$ such that $\theta^{-k}(U) \cap V \neq \emptyset$
for all $k > n$.

We now fix a H\"older continuous function $\varphi: \Sigma \to \R$ and refer to it as \emph{potential}. The Ruelle operator associated to $\theta$ and $\varphi$ is defined by its action on positive functions through
\[ \mathcal{L}(f)(\xi) := \sum_{\theta \tilde \xi=\xi} e^{\varphi(\tilde \xi)} f(\tilde \xi),\quad \hbox{ for }  f : \Sigma \to [0,\infty), \xi \in \Sigma. \]
In this setting, it is well-known that the action of $\mathcal{L}$ on Hölder continuous functions and its dual $\mathcal{L}^\ast$, defined by $ \int  f \dd \mathcal{L}^\ast(\mu)  = \int  \mathcal{L}(f) \dd \mu$ satisfies a Perron-Frobenius theorem.

\begin{theorem} \label{theo:rpf-for-R-expanding} If $\theta$ is Ruelle expanding, then $\mathcal{L}(f)$ acts boundedly on the space of H\"older continuous functions. Furthermore, if $\theta$ is topologically mixing, then there is a unique $\lambda > 0$, a unique, strictly positive and H\"older continuous function $h$ and a unique Radon measure $\mu$ such that $  \mathcal{L}(h) =\lambda h $ and
$  \mathcal{L}^\ast(\mu) =\lambda \mu$.
 Furthermore,  $\mu(U)>0$ for any open set $U \subset \Sigma$.
\end{theorem}

We would like to point out that this version of the Ruelle-Perron-Frobenius theorem is only a starting point for the theory of thermodynamic formalism for expanding and hyperbolic dynamical systems (see, e.g., \cite{Parry-Pollicott--Zeta-Functions-And-The---1990}). However, it suffices for our purposes here, as it allows us to normalize a potential function in the following sense. Recall that $\varphi$ is referred to as \emph{normalized} if the associated  operator satisfies $\mathcal{L}(\mathbf{1}) = \mathbf{1}$. On the other hand, under the assumptions of the theorem above, the existence of $\lambda$ and $h$ allows us to define
\[ \overline{\varphi} :=  \varphi + \log h - \log h\circ \theta - \log \lambda,
\]
which is a H\"older continuous and normalized potential, and  corresponds to a conjugation of $\lambda^{-1}\mathcal{L}$ with a multiplication operator.

\subsection{Skew products with a Ruelle expanding base}

We fix a further compact and metric space $(X,d_X)$ and a  map $\kappa: \Sigma \to \mathrm{Hom}(X)$, $\xi \to \kappa_\xi$, with $\mathrm{Hom}(X)$ referring to the space of homeomorphisms of $X$. We are now interested in a Perron-Frobenius theorem of the skew product
\[
T : \Sigma \times  X \to \Sigma \times  X, (\xi,x) \mapsto (\theta( \xi), \kappa(\xi)(x)).
\]
In order to introduce the relevant notations for the skew product, it is now necessary to recall some facts. As it is well-known, if $\theta$ is  \emph{$(a,\rho)$-Ruelle-expanding}, then  $\theta^n$ is  \emph{$(a,\rho^n)$-Ruelle-expanding}. In particular, for any $n \in \N$, $\xi^\ast \in \Sigma$ and $\xi = \theta^n(\xi^\ast)$, the inverse branch, denoted by
\[
\theta^{-n}_{\xi^\ast} : B_a( \xi )  \to   B_{\rho^n a}(\xi^\ast), \eta \mapsto \eta^\ast,
\]
where $\eta^\ast$ is the unique element in $B_{\rho^n a}(\xi^\ast)$ such that $\theta^n \eta^\ast  =\eta$, is well defined.  It then follows from standard arguments that $\theta^{-n}_\xi : B_a( \xi)  \to  \theta^{-n}_\xi (B_a( \xi))$ is a homeomorphism.

We proceed with giving the details of the underlying continuity assumptions. If $\alpha>0$, then  a map $f : A \to B$ between two metric spaces is \emph{$\alpha$-Hölder continuous} if there exists $D_\alpha(f) > 0$ such that  $d_B(f(x),f(y)) \leq D_\alpha d_B(x,y)^\alpha$.
From now on, we assume that $\Sigma$ and $X$ have diameter 1, that is, $\sup_{\xi_1,\xi_2} d_\Sigma(\xi_1,\xi_2) =1 $ and  $\sup_{x_1,x_2} d_X(x_1,x_2) =1$. Furthermore, with $\alpha \leq 1$
referring to the Hölder continuity of $\varphi$, define
\begin{equation}\label{eq:metric-of-the-skew-product}
d_{\Sigma \times X} ((\xi_1,x_1),(\xi_2,x_2)) :=  \tfrac{1}{2} \left(  d_{\Sigma} (\xi_1, \xi_2 )^\alpha + d_{X} (x_1,x_2)^\alpha  \right).
\end{equation}
We then say that $T$ has \emph{continuous inverse branches} if
for each pair $(\xi,\tilde{\xi})$ with $\theta \tilde{\xi} = \xi$, the map $B_a(\xi) \times X \to X$, $(\xi,x) \mapsto (\kappa(\xi)^{-1}(x))$ is continuous. Furthermore, we say that the skew product $T$ has \emph{Lipschitz continuous inverse branches}, if there exists $L > 0$ such that for each pair $(\xi_i)$ with $d(\xi_1,\xi_2) < a$, each pair $(\eta_i)$ with  $\theta(\eta_i) = \xi_i$ and $x_1, x_2 \in X$, \\
\begin{equation}
 \label{eq:Backward-Lipschitz}
d_{\Sigma \times X} ((\eta_1,\kappa(\eta_1)^{-1}(x_1)),(\eta_2,\kappa(\eta_2)^{-1}(x_2))) <   L d_{\Sigma \times X} ((\xi_1,x_1),(\xi_2,x_2)).
 \end{equation}

In order to have an effective notation at hand, define
\[\kappa_n(\xi) : =   \kappa(\theta^{n-1} \xi) \cdots  \kappa({\theta \xi}) \circ \kappa(\xi),
\quad  \varphi_n := \sum_{k=0}^{n-1}\varphi \circ \theta^k,
\]
where $\varphi : \Sigma \to \R$ is a fixed potential. In analogy to the above, the Ruelle operator is now defined by
\[ \mathcal{L}^n(f)(\xi,x) := \sum_{T^n(\eta,y) = (\xi,x)}  e^{\varphi_n (\eta)} f(\eta,y).  \]
As $\mathcal{L}$ leaves invariant the cone of positive and  continuous functions,  $\mathcal{L}^\ast$ acts on the space of finite measures through  $\int f \dd \mathcal{L}^\ast(m) := \int \mathcal{L}(f) \dd m $ and, as $\mathcal{L}(\mathbf{1}) = \mathbf{1}$, $\mathcal{L}^\ast$ acts on the set of probability measures.

\subsection{Contraction properties of $\mathcal{L}^\ast$}

The following definition now provides us with a weak form of topological exactness and expansion in average in the fiber. As we frequently will consider pairs of preimages, we will make use of the following notation, for $n\in \N$ and $\xi_1,\xi_2 \in \Sigma$ with $d_\Sigma(\xi_1,\xi_2)< a$:
\begin{align*}
\theta^n((\xi_i^\ast)) = (\xi_i)  : \iff & \theta^n \xi_i^\ast
=   \xi_i (i =1,2), d_\Sigma(\xi_1^\ast,\xi_2^\ast) \leq \rho^n   d(\xi_1,\xi_2) \\
 \iff &  \theta^n \xi_1^\ast =\xi_1, \xi_2^\ast = \theta_{\xi_1^\ast}^{-n}(\xi_2).
\end{align*}

\begin{definition}[Backward contraction in average] \label{def:LEEA} Assume that $(\Sigma,\theta)$ is $(a,\rho)$-Ruelle expanding, $\kappa: \Sigma \to \mathrm{Hom}(X)$ and that $\varphi$ is a $\alpha$-Hölder continuous and normalized potential.
\begin{enumerate}
 \item We say that $(\Sigma,\theta,\varphi,X,\kappa)$ has \emph{communicating preimages (cp)} if there exists $m_0$ such that, for all  $\xi, \eta \in \Sigma$ and $x,y \in X$, there exist  $\xi^\ast, \eta^\ast \in \Sigma$ with $\theta^{m_0}(\xi^\ast) = \xi$ and $\theta^{m_0}(\eta^\ast) = \eta$ such that
   $d_\Sigma(\xi^\ast,\eta^\ast)< a$  and  $d_X\left(\kappa_{m_0}(\xi^\ast)^{-1}(x), \kappa_{m_0}({\eta^\ast})^{-1}(y) \right)  <  a$.
 \item  We say that $(\Sigma,\theta,\varphi,X,\kappa)$ is \emph{locally eventually backward contracting in average (lebca)} if there exists $\rho^\ast < 1$ such that, for all $\xi_1, \xi_2 \in \Sigma$ and $x_1,x_2 \in X$ with $d_\Sigma(\xi_1,\xi_2)< a$ and $d_X(x_1,x_2)< a$,
\begin{align} \nonumber
  \sum_{\theta^{n_0} ((\xi_i^\ast)) = (\xi_i)} e^{\varphi_{n_0}(\xi_1^\ast)}
  \left(d_\Sigma(\xi_1^\ast,\xi_2^\ast)^\alpha + d_X\left(\kappa_{n_0}(\xi_1^\ast)^{-1}(x_1), \kappa_{n_0}(\xi_2^\ast)^{-1}(x_2) \right)^\alpha
 \right) \\ < \rho^{\ast} \left(d_\Sigma(\xi, \eta)^\alpha +  d_X(x,y)^\alpha\right) \label{eq:lecba}
.\end{align}
\end{enumerate}
If  (cp) and (lebca) hold, then $(\Sigma,\theta,\varphi,X,\kappa)$ is referred to as \emph{backward contracting in average}.
\end{definition}

At this point, it is worth noting that the (lebca)-condition is a contraction in average for the $\alpha$-Hölder distance. However, if the condition \eqref{eq:lecba} holds with respect to the exponent 1, Jensen's inequality implies that the (lebca)-condition holds with respect to the factor $(\rho^\ast)^\alpha$, provided that $\alpha \leq 1$.

We now recall the definition of the Wasserstein distance which is compatible with the weak convergence of probability measures. So assume that $m_1,m_2$ are  Borel probability measures defined on the metric space $(\Sigma \times X, d_{\Sigma \times X})$. A \emph{coupling} of $m_1$ and $m_2$ is a Borel probability measure $Q$ on  $(\Sigma \times X)^2$ such that its marginal distributions $Q \circ \pi_i^{-1}$ are equal to $m_i$, $(i=1,2)$. Furthermore, if the total mass of $Q$ is smaller than one, and  $Q \circ \pi_i^{-1} \leq  m_i$, $(i=1,2)$, we refer to $Q$ as a \emph{subcoupling}. The Wasserstein distance $W$ is now defined as
\[
W(m_1,m_2) := \inf \left\{ \int d_{\Sigma\times X}(z_1,z_2) \dd  P(z_1,z_2) :
P \hbox{ is a coupling of } m_1, m_2 \right\}.
\]
The main theorem of this section is a contraction result of $\mathcal{L}^\ast$ with respect to this metric.

\begin{theorem} \label{theo:Wasserstein}
 Assume that $(\Sigma, \theta)$ is Ruelle expanding, that $\varphi$ is $\alpha$-Hölder continuous for some $\alpha \leq 1$, that $\varphi$ is a normalized potential and that $\kappa: \Sigma \to \mathrm{Hom}(X)$ is a map such that $T$ is continuous and $(\Sigma,\theta,\varphi,X,\kappa)$ is backward contracting in average. Then there exist $\ell_0 \in \N$,   $\tau  \in (0,1)$ and $C > 0$ such that for any pair of Borel probability measures on $\Sigma \times X$ and $k \in \N$
 \[  W \left(  ( \mathcal{L}^\ast)^{k\ell_0} (m_1) ,( \mathcal{L}^\ast)^{k\ell_0} (m_2)  \right)   <   C  \tau^{k} W(m_1,m_2).\]
 If, in addition, $T$ has Lipschitz continuous inverse branches, then we may assume that  $\ell_0 = 1$.
 \end{theorem}

\begin{proof}
Throughout this proof, assume that $a > 0$, $\rho, \rho^\ast < 1$ and $m_0,n_0$ are as in the definitions of Ruelle expanding, (cp) and (lebca). Furthermore, it follows from the continuity of $T$, that $T$ has continuous inverse branches. It then follows from compactness that the inverse branches are equicontinuous. By combining (cp) and (lebca), it is now easy to see that we may assume that $m_0 = \ell n_0$, that is, $m_0$ is a multiple of $n_0$. In addition, for ease of notation, we assume that  $\rho^\ast =  \rho^{\alpha n_0}$. For $\xi_1,\xi_2 \in \Sigma$, set $d_\Sigma((\xi_i)) := d_\Sigma(\xi_1,\xi_2)$.

It now follows exclusively from the assumptions on $\theta$ and $\varphi$ that, if $d_\Sigma((\xi_i))< a$, then, for each pair $((\xi_i^\ast))$  with $\theta^n((\xi_i^\ast)) = (\xi_i)$ and with $C_\varphi :=  (e -1)D_\alpha(\varphi)/(1-\rho)$,
\begin{align}
\label{eq:equihoelder}
|\varphi_n(\xi_1^\ast) - \varphi_n(\xi_2^\ast)| & \leq \sum_{k=1}^n D_\alpha(\varphi) \rho^k d_\Sigma((\xi_i))^\alpha \leq D_\alpha(\varphi) d_\Sigma((\xi_i))^\alpha / (1-\rho),\\
\nonumber
e^{|\varphi_n(\xi_1^\ast) - \varphi_n(\xi_2^\ast)| } -1 & \leq  C_\varphi  d_\Sigma((\xi_i))^\alpha.
\end{align}

\medskip
\noindent\textsc{Step 1: Eventual contraction in average.}
In this step, we derive an estimate from condition (lebca) for the weight of those pairs of preimages whose second coordinates at least once have a distance bigger than $a$ until time $k$.
To formalize that, assume that $d_\Sigma((\xi_i))< a$  and $d_X((x_i))< a$. For $k \in \N$, set
\begin{align*}
\mathcal{C}_k & := \left\{ (\xi^\ast_i)_{i=1,2} : \theta^{kn_0}(( \xi^\ast_i)) = (\xi_i),d_X(( \kappa_{kn_0}(\xi^\ast_i)^{-1}(x_i))
 ) \geq a, \right. \\
& \phantom{ := \left\{ (\xi^\ast_i)_{i=1,2} : \right.}
\left. d_X(( \kappa_{jn_0}(\theta^{(k-j)n_0}\xi^\ast_i)^{-1}(x_i))
 ) <  a \hbox{ for } j= 1, \ldots, k-1 \right\},\\
\mathcal{D}_k & := \left\{ (\xi^\ast_i)_{i=1,2} : \theta^{kn_0}(( \xi^\ast_i)) = (\xi_i),
 d_X(( \kappa_{jn_0}(\theta^{(k-j)n_0}\xi^\ast_i)^{-1}(x_i))
 ) <  a  \hbox{ for } j= 1, \ldots, k\right\},\\
\mathcal{E}_k & := \left\{ (\xi^\ast_i)_{i=1,2} : \theta^{kn_0}(( \xi^\ast_i)) = (\xi_i), \exists j \in \{1,\ldots,k\} \hbox{ s.t. }
 d_X(( \kappa_{jn_0}(\theta^{(k-j)n_0}\xi^\ast_i)^{-1}(x_i))
 ) \geq  a\right\} .
\end{align*}
It follows from condition (lebca) that
\[
  a^\alpha \sum_{(\xi_i^\ast) \in  \mathcal{C}_1}
  e^{\varphi_{n_0}(\xi_1^\ast)}
  \leq    \sum_{\theta^{n_0}(( \xi^\ast_i)) = (\xi_i)}
  e^{\varphi_{n_0}(\xi_1^\ast)} d_X\left(
  ( \kappa_{n_0}(\xi^\ast_1)^{-1}(x_i))
   \right)^\alpha  < \rho^{\ast} \left(d_\Sigma((\xi_i))^\alpha + d_X((x_i))^\alpha \right).
\]
If $(\xi_i^\ast) \in  \mathcal{C}_{k}$ and $x^\ast_i = \kappa_{kn_0}(\xi_i^\ast)^{-1}(x_i)$, then we write $(\xi_i^\ast,x^\ast_i) \in  \mathcal{C}_{k}^\ast$, and $(\xi_i^\ast,x^\ast_i) \in  \mathcal{D}_{k}^\ast$, respectively.
Moreover, in case we need to take care of the dependence of $\mathcal{C}_k$ from the pair $((\xi_i,x_i))$, we will write $\mathcal{C}_k((\xi_i,x_i))$. It now follows from the above that
\begin{align*}
\sum_{(\xi_i^\ast) \in \mathcal{C}_k} e^{\varphi_{k n_0}(\xi_1^\ast)}
   & = \sum_{(\xi_i^\ast,x_i^\ast) \in  \mathcal{D}_{k-1}^\ast} e^{\varphi_{(k-1) n_0}(\xi_1^\ast)}
    \sum_{(\eta_i) \in \mathcal{C}^\ast_1(\xi_i^\ast,x_i^\ast)}  e^{\varphi_{n_0}(\eta_1)},
\\ & \leq
    a^{-\alpha} \rho^{\ast} \sum_{_{(\xi_i^\ast,x_i^\ast) \in  \mathcal{D}_{k-1}^\ast}} e^{\varphi_{(k-1) n_0}(\xi_1^\ast)}  \left(d_\Sigma((\xi_i^\ast))^\alpha + d_X((x_i^\ast))^\alpha \right)
\\  & \leq
      a^{-\alpha}(\rho^{\ast})^2 \sum_{(\xi_i^\ast,x_i^\ast) \in  \mathcal{D}_{k-2}^\ast}
      e^{\varphi_{(k-2) n_0}(\xi_1^\ast)} \left(d_\Sigma((\xi_i^\ast))^\alpha + d_X((x_i^\ast))^\alpha \right)
\\ & \leq
      a^{-\alpha} (\rho^{\ast})^k  \left(d_\Sigma((\xi_i))^\alpha + d_X((x_i))^\alpha \right).
\end{align*}
Using the property that the potential is normalized, it now follows from the last estimate that
\begin{align} \label{eq:measure_of_bad_part}
\sum_{(\xi_i^\ast) \in \mathcal{E}_k} e^{\varphi_{k n_0}(\xi_1^\ast)} \leq \sum_{j \geq 1}  a^{-\alpha} (\rho^{\ast})^k d_X((x_i))^\alpha \leq \left(a^\alpha(1-\rho^{\ast})\right)^{-1}  \left(d_\Sigma((\xi_i))^\alpha + d_X((x_i))^\alpha \right).
\end{align}

\medskip
\noindent\textsc{Step 2: Local contraction.}
Assume that $k \in \N$,  $d_\Sigma((\xi_i))< a$ and $d_X((x_i))< 4 (C_\varphi + a(1-\rho^\ast))$. Furthermore, define
\begin{align*}
Q & := \sum_{(\xi_i^\ast,x_i^\ast) \in \mathcal{D}_k^\ast}
e^{\min_i   \varphi_{kn_0}(\xi_i^\ast) }
\Delta_{
\left(\xi_1^\ast, x_1^\ast\right),\left(\xi_2^\ast, x_2^\ast \right)}.
\end{align*}
It now follows from \eqref{eq:equihoelder} and \eqref{eq:measure_of_bad_part} that, with
$C^\ast := 2\left( ({a^\alpha(1-\rho^\ast)})^{-1} +  C_\varphi \right)$,
\begin{align*}
1 - Q(X^2) & \leq \tfrac{1}{a^\alpha(1-\rho^\ast)} \left(d_\Sigma((\xi_i))^\alpha + d_X((x_i))^\alpha \right)  +  C_\varphi d((\xi_i))^\alpha
\\ & \leq  \left( \tfrac{1}{a^\alpha(1-\rho^\ast)}  \right) + C_\varphi \left( d((\xi_i))^\alpha + d((x_i))^\alpha
 \right)
 =  C^\ast d_{\Sigma \times X}((\xi_i,x_i)).
\end{align*}

Furthermore, the projections $\pi_i$ ($i=1,2$) to the first and the second coordinate, respectively, satisfy
 \[
 Q \circ \pi_i^{-1}  \leq  \sum_{ (\xi_i^\ast,x_i^\ast) \in \mathcal{D}_k^\ast} e^{\varphi_{kn_0} ( \xi_i^\ast)} \Delta_{\left(\xi_i^\ast,x_i^\ast \right)} = (\mathcal{L}^{k n_0})^{\ast}(\Delta_{(\xi_i,x_i)}) .
 \]
Hence, $Q$ is a subcoupling of $(\mathcal{L}^{kn_0})^{\ast}(\Delta_{(\xi_i,x_i)})$ ($i=1,2$). As it is well-known, it is possible to extend $Q$ to a coupling of these measures by adding the subcoupling
\[ R = \left(1 - Q((\Sigma \times X)^2)\right)^{-1} \bigotimes_{i=1,2}  \left( m_i -  Q \circ \pi_i^{-1} \right) \]
of mass smaller than $C_1 d(\xi_1,\xi_2)^\alpha$. For
\[ d^\dagger((\xi_1,x_1),(\xi_2,x_2)) := \min \left\{ 1, 2C^\ast  d_{\Sigma \times X}\left((\xi_1,x_1),(\xi_2,x_2)\right)
\right\}
\]
and  for $(\xi_i,x_i)$ with $d^\dagger((\xi_i,x_i)) < 1$, this now implies that
\begin{align*}
   \int d^\dagger \dd(Q+R)
  \leq & \sum_{ (\xi_i^\ast,x_i^\ast) \in \mathcal{D}_k} e^{\varphi_{kn_0}(\xi_1^\ast)}
  C^\ast    \left( d((\xi^\ast_i))^\alpha + d(  (x^\ast_i))^\alpha \right)   + R(\Sigma \times X)^2 \\
 \leq  & C^\ast \left(  \rho^{\alpha kn_0} d((\xi_i))^\alpha + \rho^{\alpha kn_0} d((x_i))^\alpha  \right) +  C^\ast   d_{\Sigma \times X}(((\xi_i,x_i)))
 \\  \leq  &  \left(\rho^{\alpha k n_0}  + \tfrac{1}{2}\right)
 d^\dagger(((\xi_i,x_i))).
\end{align*}

\medskip
\noindent\textsc{Step 3: Global contraction.}
For  $(\xi_i,x_i)$ with $d^\dagger((\xi_1,x_1),(\xi_2,x_2)) = 1$, we proceed as follows. By (cp), there is $m_0$ such that for some pair $(\xi_i^\ast)$ with $\theta^{m_0}((\xi_i^\ast)) = (\xi_i) $, we have that
$d_\Sigma((\xi_i^\ast))< a$  and  $d_X\left((\kappa_{m_0}(\xi_i^\ast)^{-1}(x_i))  \right)  < a$.
Thereafter, one applies assumption (lebca) in order to choose a pair $(\xi_i^\ast)$ with $\theta^{n_0 + m_0}((\xi_i^\ast)) = (\xi_i) $, $d_\Sigma((\xi_i^\ast))< \rho^{n_0}a$,  and  $d_X\left((\kappa_{n_0 + m_0}(\xi_i^\ast)^{-1}(x_i))\right)  < \rho^{\alpha n_0} a$.
Using the assumption that $m_0$ is a multiple of $n_0$, it follows by induction that there exists $\ell$ and a pair
$(\xi_i^\ast)$ with $\theta^{\ell n_0}((\xi_i^\ast)) = (\xi_i) $ such that
\[
d^\dagger \left( ( \xi_1^\ast , \kappa_{\ell n_0 }(\xi_1^\ast)^{-1}(x_1)) ,  (\xi_2^\ast , \kappa_{\ell n_0}({\xi_2^\ast})^{-1}(x_2) ) \right)   < 1/2.
\]
Observe that, by compactness, $C_\varphi^\ast := \inf_\xi \exp \varphi_{\ell n_0}(\xi) > 0$. It now follows in analogy to the above that
\[Q := C_\varphi^\ast    \Delta_{( \xi_1^\ast , \kappa_{\ell n_0 }(\xi_1^\ast)^{-1}(x_1)) ,  (\xi_2^\ast , \kappa_{\ell n_0}({\xi_2^\ast})^{-1}(x_2) )}  \]
is a subcoupling, which can be extended to a coupling by adding $R$ with $R(\Sigma\times X)^2 = 1 - C_\varphi^\ast$.
Hence,
\[
 \int d^\dagger \dd(Q+R) \leq  C_\varphi^\ast/2 +  R(\Sigma\times X)^2 = ( 1 - C_\varphi^\ast / 2) d^\dagger((\xi_1,x_1),(\xi_2,x_2)) .
\]

\medskip
\noindent\textsc{Step 4: Iteration.}
By eventually increasing $\ell$, we may assume that $\rho^{\alpha \ell n_0}  + {1}/{2} <  1 - C_\varphi^\ast / 2$.  Hence, for $\tau:= 1 - C_\varphi^\ast / 2$ and any pair $((\xi_i,x_i))$, the above two steps imply the existence of a coupling $P^{\ell n_0}_{((\xi_i,x_i))}$ of $(\mathcal{L}^\ast(\Delta_{((\xi_i,x_i))}))$ such that
\[
\int d^\dagger \dd P^{\ell n_0}_{((\xi_i,x_i))} \leq \tau   d^\dagger((\xi_i,x_i)).
\]
Now assume that $m_1,m_2$ are two probability measures on $\Sigma \times X$, and that $P_{(m_i)}$ is a coupling of those measures. It is now easy to check that, for $k \in \N$,
\[
P^k_{(m_i)} :=  \int \dd P^{\ell n_0}_{(z_i^{(k)})}  \cdots  \dd P^{\ell n_0}_{(z_i^{(2)})}((z_i^{(3)})) \dd P^{\ell n_0}_{(z_i^{(1)})}((z_i^{(2)})) \dd  P_{(m_i)}((z_i^{(1)}))
\]
is a coupling of $((\mathcal{L}^\ast)^{k\ell n_0}(m_i))$. Moreover,
\begin{align*}
\int d_{\Sigma \times X} \dd P^k_{(m_i)} \leq \int d^\dagger \dd P^k_{(m_i)} \leq \tau^k \int
d^\dagger \dd P_{(m_i)} \leq 2 C^\ast \tau^k  \int d_{\Sigma \times X} \dd  P_{(m_i)}.
\end{align*}
The first part of the theorem now follows with respect to $\ell_0 := \ell n_0$ and $C = 2 C^\ast$.

\medskip
\noindent\textsc{Step 5: Lipschitz continuity.} We now show that $\mathcal{L}^\ast$ is Lipschitz continuous, provided that $T$ has Lipschitz continuous inverse branches.
In order to do so, for $d_\Sigma ((\xi_i))< a$ and $x_1,x_2 \in X$, set
\begin{align*}
Q & := \sum_{T((\xi_i^\ast,x_i^\ast)) = (\xi_i,x_i)}
e^{\min_i   \varphi(\xi_i^\ast) }
\Delta_{
\left(\xi_1^\ast, \kappa( \xi_1^\ast)^{-1}(x_1)\right),\left(\xi_2^\ast, \kappa( \xi_2^\ast)^{-1}(x_2)\right)}.
\end{align*}
By adding an adequately chosen subprobability measure $R$ on $\Sigma \times X$ as in Step 2, one then obtains a coupling of $\mathcal{L}^\ast(\Delta_{(\xi_i,x_i)})$. With $L$ referring to the maximum of the constants as in \eqref{eq:Backward-Lipschitz}, one then obtains that
\begin{align*}
  2 \int d_{\Sigma \times X} \dd (Q+R)
\leq  &
  \sum_{\theta((\xi_i^\ast)) = (\xi_i)} e^{\min_i   \varphi(\xi_i^\ast) }
  \left(   d_\Sigma((\xi_i^\ast))^\alpha  + d_X((\kappa( \xi_i^\ast)^{-1}(x_i)))^\alpha \right) + R((\Sigma \times X)^2)
\\ \leq &
  \sum_{\theta((\xi_i^\ast)) = (\xi_i)}  e^{\min_i   \varphi(\xi_i^\ast) }
  \left( d_\Sigma((\xi_i^\ast)) + \left(  d_X((\kappa(\xi_1^\ast)^{-1}(x_i))) +   d_X((\kappa(\xi_i^\ast)^{-1}(x_2))) \right)^\alpha \right)
\\  &   + R((\Sigma \times X)^2)
\\ \leq &
  \rho^\alpha  d_\Sigma((\xi_i))^\alpha + 2L  d_{\Sigma \times X}((\xi_i,x_i)) +  C_\varphi d_\Sigma((\xi_i))^\alpha.
\end{align*}
Observe that this estimate implies that the action of  $\mathcal{L}^\ast(\Delta_{(\xi_i,x_i)})$ on Borel probability measures is Lipschitz continuous. The theorem follows from this with respect to the contraction rate $\sqrt[\ell n_0]{\tau}$ and the constant
$C (1 + L + C_\varphi)^{\ell n_0}$.
\end{proof}

\subsection{Ruelle's theorem for the skew product}

Theorem \ref{theo:Wasserstein} has the following immediate consequences. By applying Banach's contraction principle, one obtains that there exists a unique probability measure $\mu$ with $(\mathcal{L}^\ast)^{\ell_0}(\mu) = \mu$. Now assume that $f:\Sigma \times X \to \R$ satisfies
$|f(z_1) -f(z_2)| \leq D_\alpha(f) d_{\Sigma \times X}(z_1,z_2)$.
It now follows from Theorem \ref{theo:Wasserstein} applied to
probability measures $\mu$ and $\Delta_{\zeta}$, for $\zeta\in \Sigma \times X$ and $k \in \N$  that there exists a coupling  $P^k_{\mu,\zeta}$ of
$(\mathcal{L}^\ast)^{k\ell_0}(\mu)$ and $(\mathcal{L}^\ast)^{k\ell_0} (\Delta_{\zeta})$
such that $\int  d  \dd P^k_{\mu,\zeta} \leq C^\ast \tau^k W(\mu,\Delta_{\zeta})$. As
$W(\mu,\Delta_{\zeta}) \leq 1$,
\begin{align*}
 \mathcal{L}^{k\ell_0}(f)(\zeta) - \int f \dd \mu
 & = \int f \dd (\mathcal{L}^\ast)^{\ell_0}(\Delta_{\zeta})  -
 \int f  \dd (\mathcal{L}^\ast)^{\ell_0}(\mu) \\
 & = \int f(z_1) - f(z_2) \dd P^k_{\mu,\zeta}(z_1,z_2) \\
 & \leq D_\alpha(f) \int d_{\Sigma \times X}(z_1,z_2) \dd P^k_{\mu,\zeta}(z_1,z_2) \leq C^\ast D_\alpha(f) \tau^k.
\end{align*}
By repeating the argument for $\Delta_{\zeta_i}$, for $\zeta_i \in \Sigma \times X$ and $i =1,2$, one obtains that
\begin{align*}
 \mathcal{L}^{k\ell_0}(f)(\zeta_1) -  \mathcal{L}^{k\ell_0}(f)(\zeta_2)
 &=  \int f(z_1) - f(z_2) \dd P^k_{\zeta_1,\zeta_2}(z_1,z_2) \\
 & \leq D_\alpha(f) \int d_{\Sigma \times X}(z_1,z_2) \dd P^k_{\zeta_1,\zeta_2}(z_1,z_2)
 \\& \leq C^\ast D_\alpha(f) \tau^k d_{\Sigma \times X}(\zeta_1,\zeta_2) .
\end{align*}
These two observations can now be formalized in the following theorem in terms of the action of $\mathcal{L}$ on Lipschitz continuous functions with respect to $d_{\Sigma\times X} = (d_\Sigma^\alpha + d_X^\alpha)/2$  as follows. In order to do so, for $f:\Sigma \times X \to \R$, define
\[
D_\alpha(f) := \sup \left\{  \frac{2 |f(\xi_1, x_1) - f(\xi_2, x_2)|}{d_\Sigma(\xi_1,\xi_2)^\alpha + d_X(x_1,x_2)^\alpha} : (\xi_1,x_1) \neq (\xi_2,x_2) \right\}, \quad
\|f\|_\alpha := \|f\|_\infty + D_\alpha(f).
\]

 \begin{theorem} \label{theo:Ruelle}
 Assume that $(\Sigma, \theta)$ is Ruelle expanding, that $\varphi$ is an $\alpha$-Hölder continuous for some $\alpha \leq 1$, normalized potential and that $\kappa: \Sigma \to \mathrm{Hom}(X)$ is a  map such that $T$ is continuous and $(\Sigma,\theta,\varphi,X,\kappa)$ is backward contracting in average. Then there exist a probability measure $\mu$,  $\ell_0 \in \N$,   $\tau  \in (0,1)$ and $C^\ast > 0$ such that for any $f: \Sigma \times X \to \R$,
 \[
 \left\|  \mathcal{L}^{k\ell_0}(f) - \textstyle \int f \dd \mu  \right\|_\alpha \leq C^\ast \tau^k D_\alpha(f).
 \]
 If $T$ has Lipschitz continuous inverse branches, then $\ell_0 =1$.
 \end{theorem}

Note that it follows from Theorem \ref{theo:Ruelle} that $\mathcal{L}^{\ell_0}$ acts on the Banach space ${ f : \|f\|_\alpha < \infty }$. Now define $\Pi(f) := \left( \int f   \mathrm{d}\mu \right) \mathbf{1}$ and $S := \mathcal{L}^{\ell_0} - \Pi$. Then, by Theorem \ref{theo:Ruelle} and the spectral radius formula, the spectral radius of $S$ is less than or equal to $\tau$. In addition, it follows that the kernel of $\Pi - \mathrm{id}$ consists of constant functions, and in particular, $\dim \ker(\Pi - \mathrm{id}) = 1$. In other words, $\mathcal{L}^{\ell_0}$ exhibits a so-called \emph{spectral gap}.

As is well known, Theorem \ref{theo:Ruelle}, or equivalently, the fact that $\mathcal{L}^{\ell_0}$ has a spectral gap, has an abundance of implications, ranging from the decay of correlations for equilibrium states to asymptotic counting and the study of $\zeta$-functions (see, e.g., \cite{Parry-Pollicott--Zeta-Functions-And-The---1990}). However, in the present context, our focus lies on approximating the partial sums by a $d$-dimensional Brownian motion, which is one the topics of the next section.

\section{Consequences of the Ruelle theorem}
\label{sec:consequences-of-Ruelle}
In this section, we present two non-standard applications of the Ruelle operator theorem. The first establishes a vector-valued almost sure invariance principle for semi-flows, while the second provides a synchronization interpretation of the underlying proof strategy.

\subsection{Vector valued almost sure invariance principles}
This section is devoted to the proof of a vector-valued almost sure invariance principle (Theorem \ref{theo:vasip}) for partial sums of Hölder observables and its applications. For the exposition and the proof of the invariance principle, we closely follow \cite{Gouezel--Almost-Sure-Invariance-Principle--AP2010}.

\begin{definition} \label{def:vasip}
We say that an $\R^d$-valued process $(A_n)$ satisfies a vector-valued almost sure invariance principle of error exponent $\lambda \in (0,1/2]$ of mean $a$ and covariance matrix  $\mathbb{V}$ if there exist $a \in \R^d$, a semi-positive-definite $d \times d$-matrix $\mathbb{V}$, a probability space $\Omega$, an $\R^d$-valued process $(A^\ast_n)$ defined on $\Omega$ and a $d$-dimensional Brownian motion $(B_t)$ with covariance matrix $\mathbb{V}$ such that
\begin{enumerate}
 \item $(A_0,A_1,A_2 \ldots )$ and $(A^\ast_0,A^\ast_1,A^\ast_2 \ldots )$ have the same distribution,
 \item $ \sum_{j=0}^{n-1} A^\ast_j = B_n + an + o(n^\lambda)$ almost surely.
\end{enumerate}
\end{definition}

The main result of this section establishes an almost sure invariance principle for the process $(f\circ T^n)$, provided that $f$ is Hölder continuous and $T$ is as in Theorems \ref{theo:Wasserstein} and \ref{theo:Ruelle}.

\begin{theorem} \label{theo:vasip}
  Assume that $(\Sigma, \theta)$ is Ruelle expanding, that $\varphi: \Sigma \to \R$ is  $\alpha$-Hölder continuous for some $\alpha \leq 1$, that $\varphi$ is a normalized potential and that $\kappa: \Sigma \to \mathrm{Hom}(X)$ is a map such that $(\Sigma,\theta,\varphi,X,\kappa)$ is backward contracting in average and has Lipschitz continuous inverse branches.
  Furthermore, assume that $\mu$ refers to the unique measure with $\mathcal{L}^\ast(\mu) = \mu$  (cf. Theorem \ref{theo:Wasserstein}) and that $f : \Sigma \times X \to \R^d$ is Lipschitz continuous with respect to the metric $d_{\Sigma\times X}$ (cf. \eqref{eq:metric-of-the-skew-product}) and the Euclidean metric on $\R^d$.

  Then the process $(f\circ T^n)$ distributed according to $\mu$ satisfies a vector-valued almost sure invariance principle of error exponent $\lambda$, for any $\lambda > 1/4$, mean $a =  \int f d\mu$ and some covariance matrix  $\mathbb{V}$.   %
\end{theorem}

\begin{proof} Let  $\tau \in (0,1)$ and $C>0$ as in Theorem \ref{theo:Wasserstein}.  In order to employ Theorem 2.1 in \cite{Gouezel--Almost-Sure-Invariance-Principle--AP2010}, it remains to show that $\mathcal{L}$ has a spectral gap (Condition (I1) in \cite{Gouezel--Almost-Sure-Invariance-Principle--AP2010}) and that the operator $\mathcal{L}_t$ defined below is continuous at $t=0$ (Condition (I2) and Proposition 2.3 in \cite{Gouezel--Almost-Sure-Invariance-Principle--AP2010}).

As the spectral gap property follows from Theorem \ref{theo:Ruelle}, it remains to show the continuity assumption. In order to do so, fix a Lipschitz continuous function $f : \Sigma \times X \to \R^d$,
and define, for $t \in \R^d$,
\[ \mathcal{L}_t (u)(\xi,x) :=  \sum_{T(\eta,y) = (\xi,x)} e^{\varphi(\eta)} e^{i \langle t, f(\eta,y)  \rangle}  u(\eta,y).\]
We now show that $t \mapsto \mathcal{L}_t$ is continuous in $0$. In order to do so, note that, for $u : \Sigma \times X \to \R$ Lipschitz and  $\|t \|_1 = |t_1| + \cdots + |t_d|\leq 1$,
\begin{align*}
| \mathcal{L}_t (u)(\xi,x) -  \mathcal{L}_0(u)(\xi,x)|
& = \left| \sum_{T(\eta,y) = (\xi,x)} e^{\varphi(\eta)}
 u(\eta,y) \left( e^{i \langle t, f(\eta,y)  \rangle} -1 \right) \right| \\
& = \left| \sum_{T^{\ell_0}(\eta,y) = (\xi,x)} e^{\varphi_{\ell_0}(\eta)}
 u(\eta,y)  \sum_{n=1}^\infty \tfrac {(i \langle t, f(\eta,y)  \rangle)^n}{n!}  \right| \\
& \leq   \mathcal{L}_t (|u|)(\xi,x)  \sum_{n=1}^\infty
\tfrac { (\|t\|_1   \| f \|_\infty )^n}{n!} \\
& \leq \left(  {\textstyle \sum_{n\geq 1} \frac{(  \|f \|_\infty)^n}{n!}}  \right) \|u\|_\infty \|t\|_1
\leq \left(e^{  \|f \|_\infty} -1\right) \|u\|_\infty \|t\|_1.
\end{align*}
Hence, $t \to \mathcal{L}_t $ is (Lipschitz) continuous with respect to the sup norm in $0$. In order to obtain an estimate for the Lipschitz coefficients, we now fix $(\xi_i,x_i) \in \Sigma \times X$ such that $d_\Sigma(\xi_1,\xi_2) < a$.
By separating the sum into three terms, one obtains from the above that
\allowdisplaybreaks
\begin{align*}
&
  \left| \mathcal{L}_t (u)(\xi_1,x_1) -   \mathcal{L}_0 (u)(\xi_1,x_1) -  \mathcal{L}_t (u)(\xi_2,x_2) +   \mathcal{L}_0 (u)(\xi_2,x_2) \right|
\\ = &
  \left| \sum_{ T((\eta_j,y_j)) = (\xi_j,x_j)} e^{\varphi(\eta_1)}
  {\sum_{n=1}^\infty \tfrac{(i \langle t, f(\eta_1,y_1)  \rangle)^n}{n!}}
  u(\eta_1,y_1)
  \right.
\\ & \left. \phantom{\sum_{T^{\ell_0}((\eta_j,y_j)) = (\xi_j,x_j)}}
  - e^{\varphi (\eta_2)} {\sum_{n=1}^\infty \tfrac{(i \langle t, f(\eta_2,y_2)  \rangle)^n}{n!}}
  u(\eta_2,y_2) \right|
\\ \leq &  \phantom{+}
  \sum_{(\eta_j,y_j)} \left|1 - e^{\varphi (\eta_2)- \varphi (\eta_1)} \right|
  \left|{\sum_{n=1}^\infty \tfrac{(i \langle t, f(\eta_1,y_1)  \rangle)^n}{n!}}   \right|e^{\varphi(\eta_1)} |u(\eta_1,y_1)|
\\ & +
  \sum_{(\eta_j,y_j)} e^{\varphi (\eta_2)} \left|\sum_{n=1}^\infty \tfrac{  \langle t, f(\eta_1,y_1)  \rangle^n  -  \langle t, f(\eta_2,y_2)  \rangle^n }{n!}  \right|   |u(\eta_1,y_1)|
\\ & +
  \sum_{(\eta_j,y_j)} e^{\varphi (\eta_2)}  \left|{\sum_{n=1}^\infty \tfrac{  \langle t,f(\eta_1,y_1)  \rangle^n}{n!}}   \right|   |u(\eta_1,y_1) - u(\eta_2,y_2)|
\\ \leq &
  \phantom{+} C_\varphi \left(e^{\|f \|_\infty} -1\right)  \|t\|_1 \|u\|_\infty d_\Sigma((\xi_j))^\alpha
\\ & +
  \sum_{(\eta_j,y_j)} e^{\varphi (\eta_2)}   \|t\|_1  | f(\eta_1,y_1) -  f(\eta_2,y_2)|
  \sum_{n=1}^\infty \tfrac{   n  \|f\|_\infty^{n-1}  }{n!}\|u\|_\infty
\\ & +
  \sum_{(\eta_j,y_j)} e^{\varphi (\eta_2)} \left(e^{  \|f \|_\infty} -1\right)  \|t\|_1    D_\alpha(u) d_{\Sigma \times X}((\eta_j,y_j))
\\ \leq &
 e^{  \|f \|_\infty}   \|t\|_1 \left(C_\varphi \|u\|_\infty d_\Sigma((\xi_j))^\alpha +  D_1(f)   d_{\Sigma \times X}((\eta_j,y_j))  +  D_\alpha(u) d_{\Sigma \times X}((\eta_j,y_j))  \right).
\end{align*}
Hence, $t \mapsto \mathcal{L}_t$ is continuous in $0$ with respect to the Lipschitz norm. In particular, (I1) and (I2) are satisfied. Furthermore, as $f$ is continuous, $f \in L^p(\mu)$, for any $p \geq 1$. The theorem now follows from Theorem 2.1 in \cite{Gouezel--Almost-Sure-Invariance-Principle--AP2010}.
\end{proof}

It is worth emphasizing that, in the above theorem, the covariance matrix $\mathbb{V}$ need not be non-singular. In our dynamical setting, however, one has the following well-known and useful dichotomy, according to whether $\mathbb{V}$ is singular or not, whose proof is included for convenience.

\begin{proposition}[Folklore] \label{prop:dichotomy} Under the assumptions of Theorem \ref{theo:vasip}, the following holds.
Either, there exists  $t \in \R^d\setminus \{0\}$ and a Lipschitz continuous function $u: \Sigma \times X \to \R^d$ such that $ \langle t , f -a \rangle =u \circ T -u$ on the support of $\mu$, or
$\mathbb{V}$ is non-singular.
\end{proposition}

\begin{proof} Assume that $t  \in \R^d\setminus \{0\}$.
By Leonov's theorem (see \cite{Aaronson-Weiss--Remarks-On-The-Tightness-Of-Cocycles--CM2000}), there is $u \in  L^2$ with $g:= \langle t , f -a \rangle = u \circ T - u$ if and only if $(\langle t , S_n(f)  - an \rangle)$ is bounded in $L^2$.

If such $u$ in $L^2$ exists, we may assume without loss of generality that $\int u \dd \mu =0$. By exactness of $T$ and the fact that $\cL_0$ acts as the transfer operator on $L^1$, it follows that
\[
\sum_{k=0}^{n} \cL_0^k(g) = u - \cL_0^n(u) \xrightarrow{n \to \infty} u
\]
in $L^1$. On the other hand, as $\int g \dd \mu =0$, it follows from Theorem \ref{theo:Ruelle} that $\sum_{k=1}^{\infty} \cL_0^k(g)$ exists and is a Lipschitz continuous function. Therefore, $u$ coincides with a Lipschitz continuous function on the support of $\mu$.

On the other hand,  if $S_n(g) = (\langle t , S_n(f)  - an \rangle)$ is unbounded in $L^2$, Leonov's theorem implies that
\[ h :=  g + \sum_{k=1}^{\infty} \cL_0^k(g)  -  \sum_{k=1}^{\infty} \cL_0^k(g) \circ T =
\sum_{k=0}^{\infty} \cL_0^k(g)  -  \sum_{k=1}^{\infty} \cL_0^k(g) \circ T
\]
is non-trivial $L^2$. We now show that $\lim_{n \to \infty}  n^{-1} E (S_ng)^2 = Eh^2$, which then implies that the asymptotic variance $\langle t,\mathbb{V} t\rangle = \int h^2 \dd \mu > 0$. In order to do so, note that $(T^{-n}(\mathcal{B}))$
is a decreasing filtration, and that the conditional expectation has the representation
$ E(\, \cdot \, | T^{-n}(\mathcal{B}) ) =  \cL_0^n(\, \cdot \,) \circ T^n $
in terms of the transfer operator. In particular, as $\cL_0(\psi\circ T) = \cL_0(\mathbf{1}) \psi = \psi$ for any  $\psi$ in $L^1$,
\begin{align*}
E( h\circ T^n | T^{-n-1}(\mathcal{B})) &
=  \cL_0^{n+1}(h\circ T^n) \circ T^{n+1} = \cL_0(h)\circ T^{n+1} =0.
\end{align*}
In other words, $(h \circ T^n)$ is a reverse martingale difference. Therefore, for $m > n \geq 0$,
\begin{align*}
\int (\psi \circ T^m) (h \circ T^n) \dd \mu
& = \int (\psi \circ T^{m-n}) h \dd \mu
 = \int \cL_0((\psi \circ T^{m-n}) h)\circ T \dd \mu \\ &
 = \int \psi \circ T^{m-n} \cL_0(h)\circ T \dd \mu =0.
\end{align*}
This then implies that, for $u := \sum_{k=1}^{\infty} \cL_0^k(g)$,
\begin{align*}
 \int (S_n g)^2 \dd \mu
 & =  \int (S_n h)^2 \dd \mu + 2 \int (S_n h) (u \circ T^n - u )  \dd \mu + \int (u \circ T^n - u )^2  \dd \mu \\
 & = n \int h^2 \dd \mu  - 2  \int (S_n h) u \dd \mu  + O(1).
\end{align*}
However,  $|| \int (S_n h) u \dd \mu \leq \sqrt{n} \|h \|_2  \|u\|_2 $ by the Cauchy-Schwarz inequality. The assertion follows from this.
\end{proof}

\begin{remark} There are several standard applications of Proposition \ref{prop:dichotomy} to  a Lipschitz continuous function $f: \Sigma \times X \to \R$.

First, if $\int f \dd \mu =0$, then either $\|S_n f\|_\infty$ is uniformly bounded, or, as a consequence of the law of the iterated logarithm,  $\limsup_n |S_n f|/\sqrt{n \log \log n} = c$ a.s. for some constant $c > 0$.

Second, if $\mu(U)> 0$ for any open set $U$, then either  $f = u - u \circ T + \int f \dd \mu$ for some Lipschitz continuous function or the Brownian motion in Theorem \ref{theo:vasip} is non-trivial.
In particular, if there are two points $z_1,z_2 \in  \Sigma \times X$ and $n > 0$ with $T^n(z_i)=z_i$, for $i=1,2$, and  $S_n f(z_1) \neq S_n f(z_2)$, then the Brownian motion in Theorem \ref{theo:vasip} is non-trivial.
\end{remark}

\subsection{A vector-valued almost sure invariance principle for semi-flows}
 We now show how to extend Theorem \ref{theo:vasip} to the following class of suspension semi-flows. Assume that $(\Sigma \times X,T)$ is as above, that $h: \Sigma \times X \to (0,\infty)$ is a $d_{\Sigma \times X }$-Lipschitz continuous function   and that
 \[
 Y :=  \{ (\xi,x,s) \in \Sigma \times X \times \R : 0 \leq s \leq  h (\xi,x) \} /_\sim,
 \]
 where $\sim$ stands for the equivalence relation defined by $(\xi,x,h(\xi,x)) \sim (T(\xi,x),0)$.
 Set $h_n := \sum_{i=0}^{n-1} h(T^i(x))$ for $n \geq 1$ and  $h_0 :=0$.  The suspension flow $\psi_t: Y \to Y$ is now  defined as the vertical flow on $Y$, that is, for $t \geq 0$ and $n = \max\{n \geq 0 : h_n(x,\xi) \leq s+ t \}$,
 \[
 \psi_t (\xi,x,s) = (T^n(\xi,x), s -  h_n(x,\xi))
 \]
 Furthermore, note that the notion of a vector-valued invariance principle in discrete time in Definition \ref{def:vasip} carries over  verbatim to semi-flows by substituting $\sum_{j=0}^{n-1} A^\ast_j$ with $\int_0^t A^\ast_s \dd s$.

\begin{theorem}\label{theo:vasip-flow} Assume that $(\Sigma,\theta,\varphi,X,\kappa)$ satisfies the same assumptions as in Theorem \ref{theo:vasip}, that $(Y,\psi_t)$ is the suspension semi-flow with Lipschitz continuous roof function $h: \Sigma \times X \to (0,\infty)$, and that $f : Y \to \R^d$ is a function such that $\overline{f}(\xi,x):= \int_0^{h(\xi,x)} f(\xi,x,s) \dd s$ is Lipschitz continuous with respect to $d_{\Sigma\times X}$.

Then the process  $(f \circ \psi_t) $, distributed according to $ \mu_h = (\int h \dd \mu)^{-1} \mu \times \mathrm{Leb}$, satisfies a vector-valued almost sure invariance principle of error exponent $\lambda$, for any $\lambda > 1/4$, mean $a =    \int f \dd \mu_h $ and some covariance matrix  $\mathbb{V}$.
\end{theorem}

\begin{proof}
It follows from Theorem \ref{theo:vasip} applied to $F= (h,\overline{f})$,  after eventually enlarging the probability space, that there is  $a \in \R^{d+1}$ and a  $d+1$-dimensional Brownian motion $(B_t)$ such that, for any $\lambda > 1/4$ and $n \in \N$,
\begin{equation}\label{eq:extended-vasip-in-the-proof-of-flows}
 \sum_{k=0}^{n-1} F\circ T^k = a n  + B_n + o(n^{\lambda}).
\end{equation}
We now use the first, strictly increasing  coordinate in order to reparametrize the remaining coordinates.
In order to do so, observe that there exists $D > 1$ such that $ 1/D < h < D$ due to compactness of $\Sigma \times X$. Hence, $h_n \asymp n$ and, in particular, for $t \in [0,\infty)$ and $\zeta \in \Sigma \times X$,
\[ n(t) = n(t,\zeta)  := \min \{ n: h_n \geq t \}\]
is a well defined function from $\Sigma \times X \to \N$. Furthermore, as $h_{n(t)-1} \leq   t \leq h_{n(t)}$, it follows that  $ 0 \leq h_{n(t)} -t \leq D$, which then implies as a consequence of \eqref{eq:extended-vasip-in-the-proof-of-flows}, with $(\ast)^{(j)}$ referring to the projection onto the $j$-th coordinate, that
\begin{equation}\label{eq:difference-n-t}
t =  h_{n(t)} + O(1) = a^{(1)} n(t) + B_{n(t)}^{(1)} + o(t^\lambda).
\end{equation}
The law of the iterated logarithm now implies that
$c_1 = \limsup_t B^{(1)}_t/\sqrt{t \log \log t} > 0$ is almost surely constant. Hence, \eqref{eq:difference-n-t} implies that
\begin{align*}
 & \limsup_{t \to \infty} \frac{|n(t) - t/a^{(1)}|}{ \left(  \max\left\{  {  n(t) \log \log n(t)}, { (t/a^{(1)}) \log \log  (t/a^{(1)}) } \right\} \right)^{1/2} }\\
\leq & \limsup_{t \to \infty} \frac{|n(t) - t/a^{(1)}|}{\sqrt{n(t) \log \log n(t)}} < \frac{2 c_1}{a^{(1)}} \quad  \hbox{ a.s.}
\end{align*}
Now set $a_T:= 2 c_1 \sqrt{ T \log \log T } / a^{(1)} $. The main result by Csörgö and Révész in \cite{Csorgo-Revesz--How-Big-Are-The-Increments-Of-A-Wiener--AP1979}  implies that there are $c^{(j)}_2 > 0$ with
\begin{align*}
 c_2^{(j)} = \limsup_{T \to \infty}
 \sup_{\genfrac{}{}{0pt}{1}{0 \leq s < t \leq T,}{t-s \leq  a_T}}
 \frac{|B^{(j)}_s - B^{(j)}_t|}{ \sqrt{ 2 a_T ( \log (a_T/T) + \log\log T) }} \quad  \hbox{ a.s.}
\end{align*}
for each $j=2, \ldots ,d+ 1$. As the denominator is of order $O(T^\lambda)$, for any $\lambda > 1/4$, it follows that
\[ | B^{(j)}_{n(t)}  - B^{(j)}_{t/a^{(1)}} | = o(t^\lambda)  \quad  \hbox{ a.s.} \]
By combining these arguments, one now obtains, for each $j=2, \ldots ,d+ 1$, that
\begin{align*}
  \sum_{k=0}^{n(t)-1} F^{(j)} \circ T^k & = a^{(j)} n(t)  + B^{(j)}_{n(t)} + o(t^{\lambda})
   =  \frac{a^{(j)}}{a^{(1)}} \left(t - B^{(1)}_{n(t)}\right)
  + B^{(j)}_{n(t)} + o(t^{\lambda}) \\
  & = \frac{a^{(j)}}{a^{(1)}}\left(t - B^{(1)}_{t/a^{(1)}}\right)
  + B^{(j)}_{t/a^{(1)}} + o(t^{\lambda}).
 \end{align*}
The assertion now follows from the simple observation that
\begin{align*}
\int_0^{t} f(\psi_s(\zeta, t_0)) \dd s &
= \int_0^{t -  h(\zeta) + t_0 } f(\psi_s(T (\zeta), 0)) \dd s + O(1)
\\ &
= \sum_{k=1}^{n} \int_{0}^{h(T^k(\zeta))}  f(T^k(\zeta), s) \dd s + O(1)
=
\sum_{k=1}^{n} \overline{f} \circ T^k + O(1)
,
\end{align*}
for $n = n(t -  h(\zeta) + t_0, T(\zeta))$.
\end{proof}

\begin{remark}

There are several almost sure invariance principles or central limit theorems for suspension flows in the literature. For instance, Denker and Philipp (\cite{Denker-Philipp--Approximation-By-Brownian-Motion--ETDS1984}) showed for a suspension of a subshift of finite type with Hölder continuous roof function that observables in $L^{2+ \epsilon}$ satisfy an almost sure invariance principle with error exponent $1/2 - \delta$, which represented the state of the art for around 20 years. Thereafter, Melbourne and Török (\cite{Melbourne-Torok--Statistical-Limit-Theorems-For-Suspension-Flows--IJM2004}) obtained a central limit theorem for a suspension semi-flow with unbounded roof function over a non-uniformly expanding base transformation. This result was later refined by Melbourne and Nicol ( \cite{Melbourne-Nicol--Almost-Sure-Invariance-Principle--CMP2005,Melbourne-Nicol--A-Vector-valued-Almost-Sure--AP2009}) to vector-valued almost sure invariance principles for semi-flows and flows, again with error exponent $1/2 - \delta$.

The novelty of our approach  here is to use the vector-valued almost sure invariance principle by Gou\"ezel and the refined modulus of continuity by Csörgö and Révész, which allows one to obtain the exponent $1/4 + \delta$.
\end{remark}

\subsection{Synchronization}
In this section, we discuss the relation of Theorem \ref{theo:Wasserstein} to synchronization. In our setting, synchronization would mean that for any pair $x,y \in X$ and $\nu$-almost every $\xi \in \Sigma$,
\[ \lim_{ n\to \infty} d_{\Sigma \times X}(T^n(\xi,x),T^n(\xi,y)) =0, \]
where $\nu$ refers to the unique equilibrium state on $\Sigma$ with respect to $\varphi$.
In order to establish a connection of the method of proof of Theorem \ref{theo:Wasserstein}, we now have to make the following additional assumption in the situation of Definition \ref{def:LEEA}.

\begin{definition} \label{def:LEEA-Plus} We say that $(\Sigma,\theta,\varphi,X,\kappa)$ always has a \emph{contracting inverse branch (ahcib)}
if there exists $m_0$ such that, for all  $\xi  \in \Sigma$ and $x,y \in X$, there exist  $\xi^\ast  \in \Sigma$ with $\theta^{m_0}(\xi^\ast) = \xi$  with
$d_X\left(\kappa_{m_0}(\xi^\ast)^{-1}(x), \kappa_{m_0}({\xi^\ast})^{-1}(y) \right)  <  a$.
\end{definition}

However, note that backward contraction in average as in Def. \ref{def:LEEA} implies that  $T$ should show a certain distance expansion in $X$  and that, in particular, one should expect that the system does not synchronize. On the other hand, as we will see below, the
natural extension of $T$ has this property.
We now recall the construction of the
natural extension $(\Sigma^\dagger,\theta^\dagger,\nu^\dagger)$ of $(\Sigma,\theta,\nu)$. Set
\begin{align*}
\Sigma^\dagger :=
\{
(\xi_n)_{n \leq 0} : \theta(\xi_n) = \xi_{n+1} \forall n <  0 )
\}, \\
\theta^\dagger: \Sigma^\dagger \to \Sigma^\dagger,
(\ldots, \xi_{-1},\; \xi_0) \mapsto (\ldots, \xi_0, \theta(\xi_0)).
\end{align*}
As can easily be seen, $(\ldots, \xi_{-1}, \xi_0) \mapsto (\ldots, \xi_{-1})$ is the inverse of $\theta^\dagger$, and, therefore, $\theta^\dagger$ is invertible. So it remains to recall the construction of the $\theta^\dagger$-invariant measure $\nu^\dagger$. In order to do so, for $\xi \in \Sigma$ and $n \in \N$, set
\[ U_n(\xi):= \theta_\xi^{-n}(B_a (\theta^n(\xi)), \]
where $a$ is the constant given by the expansion in the sense of Ruelle of $\theta$, and $\theta_\xi^{-n}$ refers to the inverse branch associated with $\xi$ and $n$.
For $A \subset U_n(\xi)$ measurable, define
 \[
 \nu^\dagger\left(\left\{ (\xi_k)_{k \leq 0} \in \Sigma^\ast :
 \xi_{-n+k} \in \theta^{k}(A) \; \forall k = 0, \ldots, n \right\}\right) := \nu(A).
 \]
It now follows from $\theta$-invariance of $\nu$ that $\nu^\dagger$ is well defined, that $\nu^\dagger$  extends to a unique $\theta^\dagger$-invariant measure and that $\theta^\dagger$ is an automorphism of $(\Sigma^\dagger, \nu^\dagger)$
(see, e.g., \cite{Aaronson--An-Introduction-To-Infinite--1997}).

For the skew product $(\Sigma \times X, T, \mu)$, the same construction applies. However, due to the structure of $T$, it suffices to consider
\[
T^\dagger:  \Sigma^\dagger \times X \to \Sigma^\dagger \times X, \; ((\xi_n),x) \mapsto (\theta^\dagger(\xi_n),\kappa(\xi_0)(x)),
\]
with inverse  $S:= (T^\dagger)^{-1}$ given by
$
S : 
((\xi_n)_{n\leq 0},x) \mapsto ((\xi_n)_{n \leq - 1},(\kappa(\xi_{-1}))^{-1}(x))
$.

\begin{theorem} \label{theo:synchronization} Assume that $(\Sigma, \theta)$ is Ruelle expanding, that $\varphi$ is $\alpha$-Hölder continuous for some $\alpha \leq 1$, that $\varphi$ is a normalized potential and that $\kappa: \Sigma \to \mathrm{Hom}(X)$ is a map such that $T$ is continuous and that $(\Sigma,\theta,\varphi,X,\kappa)$ is backward contracting in average and always has a contracting inverse branch.
Then, there exists $\tau < 1$ such for that any $x,y \in X$ there exists $N : \Sigma^\dagger \to \N$  such that for $\nu^\dagger$-almost every $(\xi_k) \in \Sigma^\ast$
\[
d_X( \kappa_n(\xi_{-n})^{-1}(x), \kappa_n(\xi_{-n})^{-1}(y) ) \leq \tau^n d_X(x,y)
\]
for all $n > N((\xi_k))$. In particular, the inverse $S$ of the natural extension of $(\Sigma \times X,T,\mu)$, where $\mu$ is given by Theorem  \ref{theo:Ruelle}, is synchronizing.
\end{theorem}

\begin{proof} The proof relies on the following simplified construction in comparison to the proof of Theorem \ref{theo:Wasserstein} as it suffices to consider pairs
$(\xi,x_i) \in \Sigma \times X$ for $i=1,2$.
For the construction in Step 2, it is therefore not necessary to add a subcoupling as
\[ P^{nk_0}_{\xi,x_1,x_2}   := \sum_{\theta^{kn_0}(\xi^\ast) = \xi} e^{\varphi_{kn_0}(\xi^\ast) }
\Delta_{(\xi^\ast,\kappa_{kn_0}(\xi^\ast)^{-1}(x_1)  ) , (\xi^\ast,\kappa_{kn_0}(\xi^\ast)^{-1}(x_2)  )}.
\]
already is a probability measure. Furthermore, the estimates on $\int d^\dagger \dd P$ also hold in this setting. For the coupling in Step 3, it is straightforward to employ (ahcib) in order to see that $P^{\ell k_0}_{\xi,x_1,x_2}$ is a coupling with the appropriate contraction. So it remains to observe that
\[
 P^{k+ \ell}_{\xi,x_1,x_2}
 =  \int  P^{k}_{\xi^\ast,x_1^\ast,x_2^\ast} \dd  P^{\ell}_{\xi,x_1,x_2}( (\xi^\ast,x_1^\ast), (\xi^\ast,x_2^\ast))
\]
It hence follows as in the proof of Theorem \ref{theo:Wasserstein} that for $C> 0$ and $\tau > 0$ as in there and all $n \in \N$,
\begin{equation} \label{eq:decay-simplified}
\sum_{\theta^{n}(\xi^\ast) = \xi} e^{\varphi_{n}(\xi^\ast) }
d_X(\kappa_{n}(\xi^\ast)^{-1}(x_1) ,\kappa_{n}(\xi^\ast)^{-1}(x_2)) ^\alpha = \int d \dd  P^{n}_{\xi,x_1,x_2} \leq C \tau^n d_X(x_1,x_2)^\alpha.
\end{equation}
In terms of the natural extension, this leads to the following. For $\eta \in \Sigma$, $n \in \N$ and $A \subset B_a(\eta)$ measurable, note that it follows from the constructions of $S$ and $\nu^\dagger$ and the conformality of $\nu$ that
\begin{align*}
S^n(\{(\xi_i): \xi_0 \in A  \}) = \bigcup_{\theta^n(\eta^\ast) = \eta} \{(\xi_i): \xi_{0} \in  U_n(\eta^\ast) \cap \theta^{-n}(A) \},\\
\nu^\dagger (\{(\xi_i): \xi_{0} \in  U_n(\eta^\ast) \cap \theta^{-n}(A) \}) =
 \int_{A} e^{\varphi_n \circ \theta_{\eta^\ast}^{-n}} d\nu
.
\end{align*}
By combining these observations, it follows from \eqref{eq:decay-simplified} that,
with $\pi_X(\xi,x):=x$,
\begin{align*}
& \int d_X( \kappa_n(\xi_{-n})^{-1}(x) \kappa_n(\xi_{-n})^{-1}(y) ) \dd \nu^\dagger
\\ =  &
\int  d_X\left( \pi_X\circ S^n ((\xi_i),x), \pi_X\circ S^n ((\xi_i),y) \right)  \dd \nu^\dagger
\leq C \tau^n d_X(x,y).
\end{align*}
Set $\Omega_n := \left\{  (\xi_i) : d_X( \kappa_n(\xi_{-n})^{-1}(x), \kappa_n(\xi_{-n})^{-1}(y))^\alpha \geq \tau^{n/2} d_X(x,y)^\alpha \right\}$. It follows from the estimate that
\begin{align*}
\tau^{n/2} \nu^\dagger\left( \Omega_n \right)
\leq \int d_X( \kappa_n(\xi_{-n})^{-1}(x), \kappa_n(\xi_{-n})^{-1}(y) )^\alpha \dd \nu^\dagger
\leq  C \tau^n d_X(x,y)^\alpha,
\end{align*}
and therefore, $\nu^\dagger \left( \Omega_n \right) \leq C  \tau^{n/2}$. It hence follows from the Borel--Cantelli Lemma that almost every $(\xi_i)$ is in at most a finite number of $\Omega_n$'s.
\end{proof}

We now relate the above result to the synchronizing results by Matias (Th 5.1 in \cite{Matias--Markovian-Random-Iterations-Of--ETDS-2022}) and Gelfert \& Salcedo (\cite{Gelfert-Salcedo--Synchronization-Rates-And-Limit--MZ-2024}). Matias considers a Markov process $(Y_n)$ with values in a Polish space $Y$, a map $\kappa: Y \to \mathrm{Homeo}(\mathbb{S}^1)$ and the cocycle defined by $\kappa_n := \kappa(Y_{n})  \circ \cdots \kappa(Y_{1})$.
If the process $(Y_n,\kappa_n)$ has the Feller property and there is no common $\kappa(y)$-invariant probability on $\mathbb{S}^1$, it then follows that a local synchronization result with an exponential rate holds. For the proof, Matias makes use of a result by Malicet (Th. F in \cite{Malicet--Random-Walks-On-Rm--CMP-2017}).

Gelfert \& Salcedo, on the other hand, assume that the $(Y_n)$ are independent and identically distributed but replace $\mathbb{S}^1$ with an arbitrary compact metric space. Their main assumptions are proximality (P) and local contraction (LC), defined as follows.
\begin{itemize}
 \item[(P)] For any pair $x,y \in X$,  $\exists \omega = (\omega_i: \omega_i \in Y) $ such that $\lim_n d(\kappa_n(\omega)(x), \kappa_n(\omega)(y)) =0$.
 \item[(LC)] There exists $q < 1$ such that, for any $x \in X$ and for almost every $\omega = (\omega_i: \omega_i \in Y)$,  there exists $\epsilon > 0$ such that for all $n \in \N$,
 \[ \mathrm{diam}(\kappa_n(\omega)(B_\epsilon(x))) \leq q^n. \]
\end{itemize}
Under these assumptions, Gelfert \& Salcedo show that, for almost every $\omega$, $x,y \in X$ and $n \in \N$ (Th. 14 and 1.4 in \cite{Gelfert-Salcedo--Synchronization-Rates-And-Limit--MZ-2024}),
\[
d(\kappa_n(\omega)(x), \kappa_n(\omega)(y)) \leq C_{\omega,x,y} q^n.
\]
In order to compare these results with Theorem \ref{theo:synchronization}, assume that $\Sigma$ is a  topologically mixing subshift of finite type equipped with a normalized H\"older potential $\varphi$ and its associated equilibrium state.  As the natural extension of $\Sigma$ is the two-sided shift, one obtains synchronization of
$ \kappa((\theta^\dagger)^{-n}(\xi))^{-1}  \circ  \cdots \circ \kappa(\xi)^{-1} $ with respect to the stationary process $((\theta^\dagger)^{-n}(\xi))$. Note that, in our setting $\kappa$ may depend on infinitely many symbols.

It is worth pointing out the differences between the conditions in \cite{Gelfert-Salcedo--Synchronization-Rates-And-Limit--MZ-2024} and the ones here. First of all, observe that proximality and local contraction are asymptotic properties whereas the existence of contracting inverse branches (ahcib) and locally eventually backward contraction in average (lebca) are conditions with a finite horizon. Moreover, it is not hard to show that in the setting of subshifts of finite type, (P) implies (ahcib) and (LC) implies (lebca).

\section{Random walks on hyperbolic groups with stationary increments} \label{sec:random-walks}
This section is devoted to the application of the above results to random walks in hyperbolic groups with not necessarily independent increments whose definition is as follows.

Assume that $\mathcal{A}$ is a finite set, that $G$ is a discrete group and that $\gamma: \mathcal{A} \to G$ is a map and that $(\Sigma,\theta)$ is a transitive subshift of finite type with alphabet $\mathcal{A}$. With these objects at hand, define
\begin{equation}
 \label{eq:random-walk-with-stationary-increments}
S: \Sigma \times G \to  \Sigma \times G, (\xi,g) \mapsto (\theta\xi, g \gamma(\xi) ),
\end{equation}
where, with a certain abuse of notation, $\gamma: \Sigma \to G$, $(x_0x_1 \ldots) \mapsto \gamma(x_0)$. Furthermore, we assume that $\varphi : \Sigma \to \R$ is a Hölder continuous and normalized potential with respect to the transfer operator. As $\theta$ is transitive, it follows from Ruelle's operator theorem that there is a unique invariant and ergodic probability measure $\nu$ such that $d\nu/d\nu\circ \theta = \exp \varphi$. In particular, this measure attributes to each passage from $g \in G$ to $h \in G$ in time $n$ a weight
\[
P_n(g,h) := \nu\left(\left\{  \xi \in \Sigma : g \gamma(\xi) \cdots \gamma(\theta^{n-1}(\xi) = h \right\}\right).
\]
Note that the time evolution in the second coordinate of the skew product shares many similarities with a random walk on \( G \). However, since \( P_{m+n}(g,h) \) may not coincide with \( \sum_z P_{m}(g,z) P_{n}(z,h) \), the increments of this stochastic process are, in general, not independent. For this reason, it may be referred to as a  {random walk with stationary increments}.
However, in order to apply Theorems~\ref{theo:Wasserstein} and~\ref{theo:vasip}, one needs to consider a  compactification of \( G \) and analyze the interplay between the skew product \( S \) on \(\Sigma \times G\) and a skew product \( T \) on \( \Sigma \times \partial G\) (cf. Def. \ref{def:boundary-map}), using Martin boundary techniques.

\subsection{The boundary of a hyperbolic group}
We recall the definition of a \emph{word-hyperbolic group}. Assume that $G$ is finitely generated  and that $\mathcal{S}$ is a fixed finite set such that $\mathcal{S}^{-1} = \mathcal{S}$ and that $\mathcal{S}$ generates $G$ as a semigroup. The word metric is then defined by
\[ d_\mathrm{w}(g,h) : =  \min \{ n \in \N : \exists s_1, \ldots, s_n \in \mathcal{S} \hbox{ with } g s_1  \cdots s_n =h \},\]
and its associated \emph{Gromov product}, for $x,y,\mathbf{o} \in G$, by
\begin{equation}\label{eq:gromov product} (x \cdot y)_\mathbf{o} := \frac{1}{2} \left( d_G(x,\mathbf{o} ) + d_G(y,\mathbf{o} ) - d_G(x,y)  \right). \end{equation}
Being hyperbolic with respect to the word metric, or \emph{word-hyperbolic} for short, means that there exists $\delta > 0$ such that
$ (x \cdot z)_\mathbf{o} \geq \min\left\{(x \cdot y)_\mathbf{o},(y \cdot z)_\mathbf{o} \right\} -\delta $
for all $x,y,z,\mathbf{o} \in G$.

Each hyperbolic group comes with a geometric compactification whose properties we recall now (see, e.g., \cite{Ghys-DelaHarpe--Sur-Les-Groupes-Hyperboliques--1990}).  A sequence $(g_n)$ in $G$ is said to \emph{converge at infinity} if  $\lim_{m,n \to \infty} (g_n \cdot g_m)_\mathbf{o} = \infty$ for  some $\mathbf{o} \in G$, and we say that  $(g_n)$ and $(h_n)$ converge to the same limit at infinity if  $\lim_{n \to \infty} (g_n,h_n)_\mathbf{o} = \infty$  for  some $\mathbf{o}\in G$. The \emph{Gromov boundary} $\partial G$ of $G$ is defined as the set of equivalence classes of this equivalence relation.

In order to define a metric on $\partial G$, for $x,y \in \partial G$, one extends the Gromov product to $\overline{G}$ through
\[ (x \cdot y)_\mathbf{o} := \sup\left\{ \liminf_{m,n \to \infty} (g_m \cdot h_n)_\mathbf{o} \;:\; g_n \to x, h_m\to y \right\}.\]
For $\lambda  \in (\sqrt[2\delta]{1/2},1)$ and $r(x,y) := \lambda^{(x \cdot y) }$, it then turns out that
\begin{equation} \label{eq:parameter_for_metric_on_Gromov_boundary}
d_{\partial G}(x,y) := \inf \left\{  \sum_{k=1}^{n-1} r(x_k,x_{k+1}) : n \in \N, x_k \in \partial G, x_1 = x, x_{n}=y \right\}
\end{equation}
defines a metric, the  \emph{visual metric} on $\partial G$. Moreover, this metric is up to an error term a function of the Gromov product as there exists $C\in (0,1]$ with $ C r(x,y)   \leq d_{\partial G}(x,y)\leq r(x,y)$ for all $x,y \in \partial G$.

\begin{definition} \label{def:boundary-map} Assume that $G$ is a word-hyperbolic group $G$ and that $(\Sigma,\theta,\varphi,\gamma,S)$ are as in \eqref{eq:random-walk-with-stationary-increments}. The boundary map is defined by
$
T:  \Sigma \times \partial G \to  \Sigma \times \partial G,  \;  (\xi,x) \mapsto (\theta \xi, \gamma(\xi)^{-1}x)$.
\end{definition}

A further relevant object is the \emph{Busemann function}, defined by
\begin{equation} \label{def:Busemann-geral}
\beta_{\mathbf{o}}(g,x) :=   d_w(\mathbf{o},g) - 2 (g \cdot x)_\mathbf{o},
\end{equation}
for $g \in G$  and $x \in \partial G$, which allows us to control the metric expansion of the action of $G$ on $\partial G$ as follows. Note that $G$ acts on $\partial G$ through $gx :=  \lim g g_n$, where $(g_n)$ is a sequence with $g_n \to x$ at infinity. Due to the interplay of the Gromov product and Busemann function (see \cite[Prop. 3.3.3 (g)]{Das-Simmons-Urbanski--Geometry-And-Dynamics-In--MSM2017}), it follows that
\begin{equation} \label{eq:expansion_for_Gromov_boundary}
(x \cdot y)_{\id} = (g^{-1}x \cdot g^{-1}y)_{\id}  + \frac{1}{2} \left( \beta_{\id}(g,x) + \beta_{\id}(g,y) \right) + \mathcal{O}(1),
\end{equation}
with $\mathcal{O}(1)$ only depending on the hyperbolicity constant $\delta$. With respect to the boundary map, these estimates imply the following. Assume that $\xi_i  \in \Sigma $ ($i=1,2$) belong to the same cylinder of length $n$ and that $x_i \in \partial G$  ($i=1,2$). Then, with $\gamma_n(\xi):= \gamma(\xi) \cdots \gamma(\theta^{n-1}\xi)$, it follows that
$\gamma_n(\xi_1) =  \gamma_n(\xi_2)$. In particular,  \eqref{eq:parameter_for_metric_on_Gromov_boundary} and  \eqref{eq:expansion_for_Gromov_boundary}  imply for $(\bar{\xi}_i,\bar{x}_i) := T^n({\xi}_i,{x}_i)$ that
\begin{align}
\label{eq:expansion in the second coordinate}
d_{\partial G}(\bar{x}_1,\bar{x}_2)
\asymp
\lambda^{- \frac12  (\beta_{\id}(\gamma_n(\xi_1),x_1) +  \beta_{\id}(\gamma_n(\xi_2),x_2))}
d_{\partial G}({x}_1,{x}_2).
\end{align}

We now recall the notion of strong hyperbolicity  (see \cite{Nica-Spakula--Strong-Hyperbolicity--GGD-2016}).  A metric $d$ on $G$ is said to be \emph{strongly hyperbolic} if there exists $\epsilon>0$ such that, for all $g,h,j\in G$,
  \[
e^{-\epsilon(g \cdot j)_{\id}}
\leq
e^{-\epsilon(g\cdot h)_{\id}}
+
e^{-\epsilon(h \cdot j)_{\id}}.
\]
Under this stronger hypothesis, the Gromov product admits a well-defined extension to the boundary. More precisely, if $x,y\in\partial G$ and $(g_n)$ and $(h_n)$ are sequences converging to $x$ and $y$, respectively, then
\[
(x\cdot y)_{\id}
=
\lim_{n\to\infty}(g_n\cdot h_n)_{\id},
\]
and the limit is independent of the choice of sequences. Consequently, the visual metric on $\partial G$ takes the particularly simple form
$
d_{\partial G}(x_1,x_2)
=
e^{-\epsilon(x_1\cdot x_2)_{\id}}
$.
Furthermore, the Busemann function in \eqref{def:Busemann-geral} with respect to a strongly hyperbolic metric can be written as
\[
\beta_{\id}(g,x) = \lim_{n \to \infty}  d(g,g_n) - d(\id,g_n),
\]
where $(g_n)$ is any sequence converging at infinity to $x$. This  allows us to show that $x \mapsto \beta_{\id}(g,x)$ is Hölder continuous and that $\beta$ satisfies the cocycle identity
\[ \beta_{\id}(gh,x) = \beta_{\id}(g,x) + \beta_{\id}(h,g^{-1}x) ,\] for any pair $g,h \in G$ and $x \in \partial G$.
Moreover, as can easily  be seen, this stronger assumption implies that \eqref{eq:expansion_for_Gromov_boundary} holds without the $\mathcal{O}(1)$-error term.   Hence, as
\[
\left(S_n\beta_{\id}(\gamma(\cdot),\cdot)\right) (\xi,x) := \sum_{k=0}^{n-1} \beta_{\id} (\gamma(\theta^k \xi),\gamma_k(\xi)(x))    = \beta_{\id}(\gamma_n(\xi),x),
\]
the estimate \eqref{eq:expansion in the second coordinate} simplifies to
\begin{align}
\label{eq:expansion in the second coordinate - strongly hyperbolic}
d_{\partial G}(\bar{x}_1,\bar{x}_2)
=
e^{ \frac{\epsilon}{2}  ( S_n\beta_{\id}  (\xi_1,x_1) + S_n \beta_{\id}( (\xi_2),x_2))}
d_{\partial G}({x}_1,{x}_2),
\end{align}
which reveals that for strongly hyperbolic groups, $\exp \beta$ plays the role of the derivative, and that $\lim_n  S_n\beta_{\id}/n $ is the Lyapunov exponent of the action of the boundary map on $\partial G$.
However, at this point we would like to put emphasis on the fact that the word metric in general is not strongly hyperbolic.

%

\subsection{Positive drift}
In this section, we show that
$\lim_n d_\mathrm{w} (\id,\gamma_n(\xi))/n > 0$ and $\lim_n \beta_{\id}(\gamma_n(\xi),x)/n > 0$ almost surely under  sufficiently flexible conditions. From now on, we assume that
\begin{enumerate}
 \item[(H$_1$)]  $(\Sigma,\theta)$ is a topologically transitive subshift of finite type,
 \item[(H$_2$)] $\varphi: \Sigma \to \R$ is a Hölder continuous and normalized potential,
 \item[(H$_3$)]  $G$ is a word-hyperbolic, non-elementary group,
 \item[(H$_4$)]  $\gamma: \mathcal{A} \to G$ is a map such that the skew product $S$ is topologically transitive.
\end{enumerate}
The \emph{Green operator} is formally defined by $\G_r := \sum_{n \geq 0} r^n \cL_S^n$, for $r > 0$, with $\cL_S$ referring to the Ruelle operator associated with $S$ and the potential $\varphi$.
In order to obtain an action of $\mathbb{G}_r$, one now proceeds as follows. For $A \subset G$, set $\mathbb{X}_A:= \mathbf{1}_{\Sigma \times A}$.
Then $\mathbb{G}_r(\X_{\id})(\xi,{\id})$ is a power series in $r$ with radius of convergence $R$. It is then relatively easy to see that $\G_r(f)$ is well defined, provided that $f$ is a Lipschitz continuous function with compact support, and that $r < R$ (for details, see \cite{Bispo-Stadlbauer--The-Martin-Boundary-Of--IJM2023}). Furthermore, as shown in \cite{Stadlbauer--An-Extension-Of-Kestens--AM2013}, $R> 1$ as $G$ is not amenable. In particular, it follows from non-amenability that $S$ cannot be conservative and ergodic, which then  implies that $\G_r(f)$ is well defined for $r \leq R$ (see \cite{Jaerisch--Recurrence-And-Pressure-For--ETDS2016,Bispo-Stadlbauer--The-Martin-Boundary-Of--IJM2023}). Moreover, \cite{Bispo-Stadlbauer--The-Martin-Boundary-Of--IJM2023} made use of the following hypothesis on symmetry.
\begin{enumerate}
 \item[(S)] There exists $C > 0$ such that $ \G_R(\mathbb{X}_g)(\xi,\mathrm{id}) =  C^{\pm 1} \G_R(\mathbb{X}_{\id})(\xi,g)$ for any $g \in G$ and  $\xi \in \Sigma$,
 .
\end{enumerate}

Recall that $(\gamma_n (\xi))$, for $\gamma_n (\xi) := \gamma(\xi) \gamma(\theta \xi)  \cdots \gamma(\theta^{n-1}\xi)$. An argument for Martin boundaries provides us with almost sure convergence of $(\gamma_n (\xi))$ (see \cite{Shwartz--Thermodynamic-Formalism-For-Transient--CMP2019}) and, as $G$ is hyperbolic, the limit is an element of $\partial G$. The next result now provides us with an upper bound for the minimal distance of $\{\gamma_m (\xi): m \in \N\}$ to the point on the geodesic half ray towards the limit point at distance $n$. We would like to remark that the result relies on  two estimates by Gouëzel (\cite[Lemmas 2.5 and 2.6]{Gouezel--Local-Limit-Theorem-For--JAMS2014}) for random walks, which were extended in \cite{Bispo-Stadlbauer--The-Martin-Boundary-Of--IJM2023} to the setting  here.

Moreover, since we only need to consider the Martin boundary for $r=1$, we use the following observation concerning the results in  \cite{Gouezel--Local-Limit-Theorem-For--JAMS2014,Bispo-Stadlbauer--The-Martin-Boundary-Of--IJM2023}. In there, condition (S)  is used only in order to obtain uniform estimates for $1 \leq r \leq R$ through Lemma 2.5 in \cite{Gouezel--Local-Limit-Theorem-For--JAMS2014} and Lemma 4.4 in \cite{Bispo-Stadlbauer--The-Martin-Boundary-Of--IJM2023}, respectively. However, by analyticity of $\G_r$ for $1 \leq r <  R$, these estimates are trivial for $r=1$ and, therefore, the results in  \cite{Bispo-Stadlbauer--The-Martin-Boundary-Of--IJM2023} which are relevant  here remain valid without requiring condition (S).

\begin{proposition}\label{prop:convergence-to-the-boundary}
Assume that (H$_1$ - H$_4$) hold. Then, the following hold.
\begin{enumerate}
 \item For $\nu$-almost every $\xi \in \Sigma$, the sequence  $(\gamma_n(\xi))$ converges  to  $\pi(\xi):= \lim_n \gamma_n(\xi) \in \partial G$.
 \item There exists $C > 0$ such that for $\nu$-almost every $\xi \in \Sigma$, any geodesic ray $s: [0, \infty ) \to G$ from $\mathrm{id}$ to $\pi(\xi)$, and for all but finitely many $n \in \N$,
\[  \min_{m \in \N} d_G(s(n), \gamma_m(\xi)) <   C{\log \log n}.\]
\item
There exists $C > 0$ such that for $\nu$-almost every $\xi \in \Sigma$, any geodesic ray $s: [0, \infty ) \to G$ from $\mathrm{id}$ to $\pi(\xi)$, and for all but finitely many $n \in \N$, there is $m_n \in \N$ such that
\[d_G(\gamma_n(\xi), s(m_n)) < C \log m_n.\]
\end{enumerate}
\end{proposition}

\begin{proof} By Theorem 4.1 in \cite{Shwartz--Thermodynamic-Formalism-For-Transient--CMP2019}, for $\nu$-almost every $\xi \in \Sigma$, the sequence $(\gamma_n(\xi))$ converges to the minimal Martin boundary, which is, under the above assumptions, homeomorphic to $\partial G$ (Theorem 5.2 in \cite{Bispo-Stadlbauer--The-Martin-Boundary-Of--IJM2023}), and therefore $(\gamma_n(\xi))$ converges at infinity, say to $\pi(\xi) \in \partial G$. This proves the first assertion.

Let us now fix $x \in \partial G$, a geodesic ray $s : [0, \infty) \to G$ from $\id$ to $x$ and set
\begin{equation}
 m_n(f) :=  \frac{\G_1(f)(\xi, s(n))}{\G_1(\X_{\id})(\xi, s(n))}. \label{eq:definition-of-m-n-xi}
\end{equation}
By Theorem 5.2 in \cite{Bispo-Stadlbauer--The-Martin-Boundary-Of--IJM2023},
$ m_x(f) := \lim_{n \to \infty} m_n(f)$
exists for any Lipschitz continuous function $f$ with compact support and the limit does not depend on $\xi$. Moreover, $m_x$ is a $\sigma$-finite measure $m_x$ with $\cL^\ast_S(m_x) = m_x$ such that $m_x$ is minimal within this class.

We now make use of an estimate in \cite{Bispo-Stadlbauer--The-Martin-Boundary-Of--IJM2023} as follows. Assume that
$  k < m $ in $\N$ and $\ell \in \N $ are such that $0 < k - \ell < k+\ell < m$. Moreover,
set $\Omega_m := \{ g \in G: (s(m) \cdot x)_{\mathbf{o}}  \leq (g \cdot x)_{\mathbf{o}} \}$ and
\begin{align*} E(\eta,g) := \left\{ (\bar{\eta},n) \in \Sigma \times \N : \exists n> 0 \hbox{ s.t. } S^n(\bar \eta,\id ) = (\eta,g),
\right.
\\
\left.
d_\mathrm{w}(\gamma_t(\bar{\eta}), s(k)) > \ell \hbox{ and }  \gamma_t(\bar{\eta}) \notin \Omega_m \forall t< n
 \right\}. \end{align*}
That is, $E(\eta,g)$ corresponds to those orbits from $\mathbf{o}$ to $\Omega_m$ which never pass through $B_\ell(s(k))$ before entering
into $\Omega_m$ for the first time. Moreover, as any orbit converging to $x$ has to pass through $\Omega_m$, it follows that for $m_x$-a.e. $(\eta,\id)$ with this property, there is a unique pair $( \bar\eta, g)$ with $(\eta,\cdot) \in E(\bar\eta,g)$. So let us define
\[
A_{k,\ell,m} := \left\{ \eta : \exists n \in \N  \hbox{ s.t. } (\eta,n) \in \textstyle\bigcup_{(\bar\eta,g) \in \Sigma \times \Omega_m} E(\bar\eta,g)  \right\}.
\]
The Hölder continuity of $\varphi$ implies that $  F(\eta,g) :=\sum_{(\bar{\eta},n) \in E(\xi,g)} \exp \varphi_n(\bar\eta)$ is locally Lipschitz continuous on its domain of definition. Hence,
\begin{align*}
 m_x\left(\left\{ A_{k,\ell,m}  \times \{\id\})  \right\}\right)   =    \int F \dd m_x
 =   \lim_{n\to \infty} \frac{\G_1(\mathbf{1}_{A_{k,\ell,m} \times \id}  )(\xi, s(n))}{\G_1(\X_{\id})(\xi, s(n))}
\end{align*}
We now use Lemma 4.5 in  \cite{Bispo-Stadlbauer--The-Martin-Boundary-Of--IJM2023} in order to obtain a double exponential bound. In order to do so, recall that it was shown in there, that there exists $\lambda_1 > 1$ such that, for any such configuration with $m$ sufficiently large,
\[
\sum_{({\eta},m) \in E(\bar{\eta},g)}
e^{\varphi_m( {\eta})} \leq 2^{-\lambda_1^\ell}.\]
However, by combining this result with the Ancona inequality (Th. 4.6 (i) in \cite{Bispo-Stadlbauer--The-Martin-Boundary-Of--IJM2023}), it follows  that there exists $C> 0$ such that for any of these configurations
\[
\sum_{(\bar{\eta},m) \in E}
e^{\varphi_m(\bar{\eta})} \leq C
2^{-\lambda_1^\ell}
\G_1(\X_{\id})(\xi, s(k - \ell))
\G_1(\X_{s(k+\ell)})(\xi,g).
\]
As $\gamma(\Sigma)$ is finite, it follows that there exists $\lambda_2 > 0$ such that $\G_1(\X_{s(k-\ell)})(\xi, s(k+\ell)) \geq C \lambda_2^{2\ell}$, by eventually increasing $C$. For $k_-=0$ and $k_+ = n$, this implies that, again by eventually increasing $C$ and for some $\lambda_3 > 1$,
\begin{align*}
 m_x\left(\left\{ A_{k,\ell,m}  \times \{\id\})  \right\}\right)
\leq &  C  \lim_{n \to \infty}
\frac{2^{-\lambda_1^\ell} \G_1(\X_{\id})(\xi, k - \ell) \G_1(\X_{s(k + \ell)})(\xi, s(n)) }{\G_1(\X_{\id})(\xi, s(n))} \\
\leq & C 2^{-\lambda_1^\ell} \lambda_2^{-2\ell} \leq C2^{-\lambda_3^\ell}.
\end{align*}
However, as the constant $C$ only depends on the uniform constant in Ancona's inequality, $\min \varphi$ and $\lambda_1$, the constant is uniform in $ x, k$ and $\ell$. Letting $m \to \infty$ now shows that
\[ m_x\left(\left\{ (\eta,\id) : \gamma_t(\eta) \notin B_\ell(s(k)) \right\}\right) < C2^{-\lambda_3^\ell}.\]
Choose $D > 0$ such that $ \sum_{n> 2} C2^{-\lambda_3^{\ell_n}} < \infty$, for $\ell_n := D \log\log n$, and define $k_2 =0$ and, by induction, $k_{n+1} := k_n + \ell_{n+1}$. It now follows from the Borel-Cantelli Lemma, that for $m_x$-a.e. $(\eta,\id)$, there is some $N \in \N$ such that $(\gamma_t(\eta))_{t > 0}$ passes through $B_{\ell_n}(s(k_n))$ for all $n > N$.
It then follows from the hyperbolicity of $G$ that this means that
$(\gamma_t(\eta))_{t > 0}$  has to pass through $B_{\ell_n}(s(m))$ for any $n>M$ and $k_{n-1}< m < k_n$. As $k_n$ grows faster than $n$, it is possible to express the above independently of $(k_n)$ and $(\ell_n)$. That is, for  $m_x$-a.e. $(\eta,\id)$, there is some $N \in \N$ such that $(\gamma_t(\eta))_{t > 0}$ passes through $B_{\ell_n}(s(n))$ for all $n > N$.

In order to provide measurability in $x$, it is necessary to rewrite the above in terms of Gromov products as follows. For $m \in \N$ define $\Omega_m^\ast := \{ g : m -d \leq (g \cdot x)_{\mathbf{o}}  \leq m+ d \}$, where $d$ is chosen such that each path from ${\mathbf{o}}$ to $x$ has to pass through $\Omega_m^\ast$. With
\[
\Omega_{m,t}^\ast\left(x\right) := \left\{ g \in \Omega_m^\ast: (\mathbf{o},x)_g > t \right\},
\]
the above implies that for  $m_x$-a.e. $(\eta,\id)$, there is some $N \in \N$ such that $(\gamma_t(\eta))_{t > 0}$ passes through $\Omega_{m,\ell_m + K}^\ast\left(x\right)$ for all $m > N$ and some constant depending on $d$ and the hyperbolicity of $G$.

As the product of $\nu$ and Haar measure is a conformal measure, this product is a convex combination of $\{m_x : x\in \partial G\}$. Hence,
for  $\mu$-a.e. $\eta$, there is some $N \in \N$ such that $(\gamma_t(\eta))_{t > 0}$ passes through $\Omega_{m,C \log \log m}^\ast\left( \pi(\eta)\right)$ for all $m > N$ and for some $C > D$. This proves the second part.

The proof of the remaining assertion follows the same lines. For $x \in \partial G$, $t \in \N$, $ {\eta} \in \Sigma$ and  $g \in \Omega_{m,t}^\ast(x)$, define
\[ E^\ast(\eta,g) := \left\{
(\bar{\eta},n) \in \Sigma \times \N : S^n(\bar{\eta},\id) = (\eta, g), \, \gamma_t(\bar{\eta}) \notin \Omega_{m,t}^\ast(x) \forall k = 0,\ldots,n-1
\right\},\]
and $F(\eta,g) := \sum_{(\bar{\eta},n) \in E^\ast(\eta,g)} \exp \varphi_n(\bar{\eta})$ whenever $E^\ast(\eta,g) \neq   \emptyset$ and $0$ else. As  $F$ is locally Hölder continuous, it follows as above that
\[
\int F \dd m_x =
m_x\left(\left\{ (\eta,\id) : \gamma_n(\eta) \in \Omega_{m,t}^\ast(x) \hbox{ for some } n \in \N \right\} \right).
\]
Furthermore, a trivial estimate and, for $m$ sufficiently large and any $\xi \in \Sigma$, the Ancona inequality imply that
\begin{align*}
\int F \dd m_x \leq & \int \X_{ \Omega_{m,t}^\ast(x)} \G_1(\X_{\id})  \dd m_x
\\ \asymp & \sum_{g \in \Omega_{m,t}^\ast(x)} \G_1(\X_{\id})(\xi, s(m))  \G_1(\X_{s(m)})(\xi, g)  m_x(\X_g)
\\
\asymp  & \sum_{g \in \Omega_{m,t}^\ast(x)}  \G_1(\X_{g} \G_1(\X_{s(m)}))   (\xi, s(m)).
\end{align*}
By Lemma 4.4 in \cite{Bispo-Stadlbauer--The-Martin-Boundary-Of--IJM2023},  $\sum_{g : d_G(s(m),g) =n}  \G_R(\X_{g} \G_R(\X_{s(m)}))   (\xi, s(m))$ is uniformly bounded in $n$, where $R> 1$ is the critical parameter of $\G_r(\X_{\id})$. It is now easy to see that this implies that there is $\lambda_4$ such that
\[
 \sum_{g \in \Omega_{m,t}^\ast(x)}  \G_1(\X_{g} \G_1(\X_{s(m)}))   (\xi, s(m)) \ll \lambda_4^t.
\]
The remaining assertion now follows as above by applying the Borel-Cantelli Lemma.
\end{proof}

\begin{remark} \label{rem:convergence-to-the-boundary-for-r}
Even though this is not directly related to the question under consideration here, it is worth noting that Proposition \ref{prop:convergence-to-the-boundary} holds in the following, more general context. As in \cite{Bispo-Stadlbauer--The-Martin-Boundary-Of--IJM2023}, we refer to a $\sigma$-finite measure $m$ on $\Sigma \times G$ as $1/r$-conformal if $\cL_S^\ast(m) = r^{-1} m$. Furthermore, it is shown there that, for each $r \in [1, R]$, the set of minimal $1/r$-conformal measures can be identified with $\partial G$ via the limit in \eqref{eq:definition-of-m-n-xi} as $n \to \infty$, replacing $\G_1$ with $\G_r$. Since the proofs of the first two assertions do not rely on Lemma 4.4 in \cite{Bispo-Stadlbauer--The-Martin-Boundary-Of--IJM2023}, they remain valid for any $1/r$-conformal measure $m$ with $r \in [1, R]$. The third assertion, however, depends on this lemma and therefore holds only for parameters $r \in [1, R)$, even in the presence of condition (S).
\end{remark}

We now show that $(\gamma_n)$ has \emph{positive drift}, i.e. $\lim_{n \to \infty} d_G(\gamma_n(\xi),\id) > 0$ almost surely. For random walks on linear groups, this is a classical theorem due to Furstenberg.

\begin{theorem} \label{theo:positive-drift}
 Assume that  (H$_1$ - H$_4$) hold. Then there exists $a> 0$ such that for $\nu$-almost every $\xi \in \Sigma$,
$\lim_{n \to \infty} d_G(\gamma_n(\xi),\id)/n =a $.
\end{theorem}

\begin{proof}
Fix $\xi \in \Sigma$ such that $x:=\pi(\xi)$ exists.  Furthermore, choose a geodesic half ray $s: [0,\infty)$ from $\id$ to $x$. For $k,n \in  \N$, set
\begin{align*}
 F_k & := \{ (\eta, g) : (g \cdot x)_{\id}  \geq k \}, \quad
 P_k  := \{ (\eta, g) : (g \cdot x)_{\id}  < k \},\\
 \mathrm{FE}^n_k & :=  S^{-n}(F_k) \cap \bigcap_{\ell =0}^{n-1} S^{-\ell}(P_k) \cap \Sigma \times \{ \id \}, \quad \mathrm{FE}_k  := \bigcup_{n \geq 0} \mathrm{FE}^n_k
\end{align*}
 It now follows from Proposition \ref{prop:convergence-to-the-boundary} that
 $\mathfrak{n}_k (\eta) := \min \{ n : \eta \in \mathrm{FE}^n_k \} < \infty$ almost surely with respect to $m_x$. We now show that
 $\int_{\Sigma \times \{ \id \}} \mathfrak{n}_k  \dd m_x \ll k$ uniformly in $\xi,\eta$ and $k$.

 In order to do so, we first have to introduce some notation. For $a = (a_0,\ldots,a_{n-1}) \in \cA^n$, set $ [a] = [a_0\ldots a_{n-1}] := \{(\eta_i) \in \Sigma : a_\ell = \xi_\ell, 0 \leq \ell < n \}$ and $|a|=n$. Moreover, set $\Omega_k:= \bigcup_{n \geq 0} S^n(\mathrm{FE}^n_k)$ and note that $\mathfrak{n}_k(\eta) = \min\{n : S^n(\eta,\id) \in \Omega_k \}$. It now follows from conformality of $m_x$ and the simple identity $\G_1 \circ \G_1  - \G_1= \sum_{n \geq 0} n\cL_S^n$ that
 \begin{align}
  \nonumber
  \int_{\Sigma \times \{ \id \}} \mathfrak{n}_k  \dd m_x &
    = \sum_{n=0}^\infty m_{\xi}( \{(\eta,\id) : \mathfrak{n}_k(\eta) = n \})
    = \sum_{n=0}^\infty \quad \sum_{|a| = n, \mathfrak{n}_k|_{[a]} =n} n m_x([a] \times \{\id \})
  \\ & \nonumber
    = \sum_{n=0}^\infty \quad \sum_{|a| = n, \mathfrak{n}_k|_{[a]} =n} n \int \mathbf{1}_{\Omega_k} \cL_S^n(\mathbf{1}_{[a] \times \id})  \dd m_x
   \\ &  \label{eq:upperbund-for-return-time}
   \leq  \sum_{n=0}^\infty   \int \mathbf{1}_{\Omega_k} n \cL_S^n(\X_{\id})  \dd m_x
    \leq  \int \mathbf{1}_{\Omega_k}   \G_1^2(\X_{\id})  \dd m_x.
 \end{align}
 As $S$ is topologically transitive and $\varphi$ is Hölder continuous, it follows that the radius of convergence $R$ of $\sum_n r^n\cL^n(\X_{\id})(\xi,g)$ does not depend on $\eta$ and $g$. Furthermore,
 as $G$ is non-amenable, it follows from \cite{Stadlbauer--An-Extension-Of-Kestens--AM2013} that $R> 1$. Hence, $\G_1^2(f) = \G_1 \circ \G_1(f)$ is well defined for any continuous function $f: \Sigma \times G \to \R$ with compact support.

 As the next step of the proof, we now show that $\G_1^2(\X_{\id})(\eta,g)/ \G_1(\X_{\id})(\eta,g)$ grows at most linearly in $d_G(\id,g)$.
 In order to do so, observe that   Ancona's inequality (see Th. 4.6 (i) in \cite{Bispo-Stadlbauer--The-Martin-Boundary-Of--IJM2023}) implies that
 \begin{align*}
  & \G_1^2(\X_{\id})(\eta, s(n)) \\
     = & \sum_{(g \cdot x)_{\id} \leq 0 }
       \G_1 \circ ( \X_g \G_1(\X_{\id})(\eta, s(n)) +
       \sum_{k=1}^{n-1} \sum_{(g \cdot x)_{\id} = k }
       \G_1 \circ ( \X_g \G_1(\X_{\id})(\eta,s(n)) \\
  & + \sum_{(g \cdot x)_{\id} \geq n }
       \G_1 \circ ( \X_g \G_1(\X_{\id})(\eta, s(n)) \\
    \asymp & \sum_{(g \cdot x)_{\id} \leq 0 }  \G_1 \circ ( \X_g \G_1(\X_{\id})(\eta, \id) \cdot \G_1(\X_{\id})(\eta, s(n)) \\
    & + \sum_{k=1}^{n-1} \sum_{(g \cdot x)_{\id} = k }
       \G_1 \circ ( \X_g \G_1(\X_{s(k)})(\eta,s(k)) \cdot \G_1(\X_{\id})(\eta, s(n)) \\
    &  +  \sum_{(g \cdot x)_{\id} \geq n }  \G_1 \circ ( \X_g \G_1( \X_{s(n)})(\eta, s(n)) \cdot \G_1(\mathbf{1}_{\Sigma \times \id})(\eta, s(n)) \\
    \ll & (n+1) \G_1^2(\X_{\id})(\eta, \id) \G_1(\X_{\id})(\eta, s(n))
    \ll n \G_1(\X_{\id})(\eta, s(n)).
 \end{align*}
 It now follows from \eqref{eq:upperbund-for-return-time}, from $m_x = \lim_n m_n$ (cf. \eqref{eq:definition-of-m-n-xi}), from Ancona's inequality and from the reverse triangle inequality in Gromov hyperbolic spaces that
  \begin{align*}
    \int_{\Sigma \times \{ \id \}} \mathfrak{n}_k  \dd m_x &
    \ll \sum_{[a]\times \{g\} \subset \Omega_k} d_G(\id,g) \G_1(\X_{\id})(\eta, g) m_x([a]\times \{g\} ) \\
    & \asymp k \sum_{[a]\times \{g\} \subset \Omega_k}  \tfrac{k + d_G(s(k),g)}{k} \;  \G_1(\X_{s(k)})(\eta, g) \;  \G_1(\mathbf{1}_{[a] \times \{g\}})(\eta, s(k))
    \\ & \leq
    k \sum_{g \in G}  \left(1 + d_G(\id,g)/k \right)  \;  \G_1(\X_{\id})(\eta, g) \;  \G_1(\X_g)(\eta, \id)
  \end{align*}
In order to show that the sum in the last line is convergent, it now suffices to note that Lemma 4.4 in \cite{Bispo-Stadlbauer--The-Martin-Boundary-Of--IJM2023}, as in the proof of  Proposition \ref{prop:convergence-to-the-boundary}, implies that  $\sum_{d_G(\id,g) =n}$ $\G_1(\X_{\id})(\eta, g)  $ $ \G_1(\X_g)(\eta, \id) $  decays exponentially in $n$. Hence, $\int_{\Sigma \times \{ \id \}} \mathfrak{n}_k \dd m_x \ll k$.

 It remains to show how to deduce the statement of the theorem from this estimate by applying Derriennic's almost subadditive ergodic theorem (see \cite{Derriennic--Un-Theoreme-Ergodique-Presque--AP1983}). First, observe that by a tree approximation (cf. \cite[Theorem 12]{Ghys-DelaHarpe--Le-Bord-Dun-Espace--1990_2}) there exists a constant $C > 0$ such that $\mathfrak{n}_k(\xi) + \mathfrak{n}_\ell (\theta^{k}\xi)  - C \leq  \mathfrak{n}_{k + \ell }(\xi)$ for almost all $\xi \in \Sigma$ and $k, \ell \in \N$. As $\int_{\Sigma \times \{ \id \}} \mathfrak{n}_k \dd m_x/k \ll 1$, it follows that the same holds with respect to $\nu$ and, in particular, that the almost subadditive ergodic theorem is applicable. Hence, there is $c < \infty$ such that
 $c = \lim_k  \mathfrak{n}_k(\eta)/k$ $\nu$-almost surely. The theorem now follows from the observation that $k \leq d_G(\id,\gamma_{\mathfrak{n}_k})$ and a further application of the  subadditive ergodic theorem.
\end{proof}

\begin{remark} \label{rem:positive-drift-for-r} In analogy to Remark \ref{rem:convergence-to-the-boundary-for-r}, observe that the main part of the proof carries over  verbatim to $1/r$-conformal measures as long as  $r \in [1, R)$. However, as the subadditive ergodic theorem is not applicable, one obtains that there is $C_r> 0$ such that for $\mathfrak{n}_k(\xi) := \min\{ n : (\gamma_{n}(\xi) \cdot \pi(\xi))_{\id}> k \}$ and any $1/r$-conformal measure $m$ with $m(\Sigma \times \{\id\}) =1$,
\[  \int_{\Sigma \times \{\id\}} \mathfrak{n}_k  \dd m \leq C_r k \quad   \forall k \in \N.\]
\end{remark}

As a consequence of Theorem \ref{theo:positive-drift}, the ergodic theorem and Proposition \ref{prop:convergence-to-the-boundary}, one now immediately obtains the following result.

\begin{corollary} \label{cor:pos-lyapunov-exponent}
 If (H$_1$ - H$_4$) hold, then,  for $\nu$-almost every $\xi \in \Sigma$,
 \[  \lim_{n \to \infty} \tfrac1n  d_G(\gamma_n(\xi),\id)  =
  \lim_{n \to \infty} \tfrac1n  \beta(\gamma_n(\xi),\pi(\xi))
  > 0.
 \]
\end{corollary}

\begin{remark} These results are known for simple random walks on hyperbolic groups, which correspond in our setting to a full shift $\Sigma$ with a locally constant and normalized potential such that, in addition, H$_3$ and H$_4$ hold: Ancona showed  almost sure convergence to the boundary (Cor. 6.3 in \cite{Ancona--Theorie-Du-Potentiel-Sur---1990}) and Kaimanovich was able to obtain positive drift (Theorem 7.3 in \cite{Kaimanovich--The-Poisson-Formula-For--AM22000}).
%
\end{remark}

\subsection{Backward contraction in average}
We are now in position to prove that $T$ is backward contracting with respect to a large class of metrics on $\partial G$. In order to do so, recall that two metrics $d_1,d_2$ on a space $Y$ are referred to as quasi-isometric if there exists constants $A\geq 1,$ , $B> 0$ such that, for any pair $x,y \in Y$,
\[A^{-1}d_1(x,y) - B \leq   d_2(x,y) \leq A d_1(x,y) + B, \]
and that Gromov-hyperbolicity is preserved under quasi-isometries.

\begin{theorem} \label{theo:hyperbolic-groups-and-lebca}
Assume that (H$_1$ - H$_4$)  hold, and that $(\Sigma,\theta)$ is topologically mixing. Furthermore, assume that $d^\ast$ is quasi-isometric to the word metric on $G$, and that $d^\ast_{\partial G}(x,y)$ is the associated, induced metric on $\partial G$.  Then there is $\alpha \in (0,1]$ such that the skew product  $T$ is backward contracting in average with respect to $(d^\ast_{\partial G})^\alpha$.
\end{theorem}

\begin{proof}
 In order to prove the theorem, one has to check that $T$ has communicating preimages (cp) and that $T$ is locally eventually backward contracting in average (lebca).

 \medskip

 \noindent\textsc{Step 1: (cp)} In order to verify this hypothesis, fix three distinct points $z_1,z_2,z_3 \in \partial G$, three geodesic half rays $g_c$ from $\id$ to $z_c$ ($c=1,2,3$) and a cylinder $[b] \subset \Sigma$. It follows from transitivity of $S$ that there exist three sequences $(\eta^{(c)}_n)$ in $\Sigma$ such that $\lim_n \gamma_n((\eta^{(c)}_n)) = z_c$ and $\theta^n{\eta^{(c)}_n} \in [b]$ for all $n \in \N$ and $c =1,2,3$. As the $z_c$ are distinct,  $\delta:= \min_{c \neq c^\ast} d^\ast_{\partial G}(z_c , z_{c^\ast})> 0$.

 So assume that $(\xi_i,x_i) \in \Sigma \times X$ for $i = 1,2$.  As $\theta$ is topologically mixing, we may assume without loss of generality that $\xi_1,\xi_2$ are in the same cylinder. Now assume that $c(i)$ is chosen such that $d^\ast_{\partial G}(x_i, z_{c(i)}) = \min_j d^\ast_{\partial G}(x_i, z_j)$. Then, for any $c  \notin\{c(1),c(2)\}$,
 $d^\ast_{\partial G}(x_i, z_{c}) \geq \delta/2$ for $i=1,2$.
 Furthermore, for $\xi^{(i)}_n$ referring to the  preimage of $\theta^n$ closest to $\eta^{(c)}_n$, it follows that $\gamma_n(\eta^{(c)}_n) = \gamma_n(\xi^{(i)}_n)$ for $i=1,2$.

It now remains to apply \eqref{eq:expansion_for_Gromov_boundary} and  \eqref{eq:expansion in the second coordinate} as follows. It follows from the choice of $z_c$ that $(z_c \cdot x_i)_{\id}$ is uniformly bounded from above, say $(z_c \cdot x_i)_{\id} \leq M$. Furthermore, by construction,
$(\gamma_n(\eta^{(c)}_n))$ stays within a bounded distance to the half ray $g_c$. As  $(z_c \cdot x_i)_{\id} \leq M$, this implies that
$(\gamma_n(\eta^{(c)}_n))$ also stays within a bounded distance to any geodesic from $x_i$ to $z_c$. On the other hand, by  \eqref{eq:expansion_for_Gromov_boundary},
\begin{align*}
(z_c \cdot x_i)_{\id}  -  (z_c \cdot x_i)_{\gamma_n(\eta^{(c)}_n)} =    \frac12 (  \beta^\ast_{\id}(\gamma_n(\eta^{(c)}_n),z_c) +  \beta^\ast_{\id}(\gamma_n(\xi^{(i)}_n),x_i)).
\end{align*}
As the left-hand side is uniformly bounded, we obtain that
\[\beta^\ast_{\id}(\gamma_n(\xi^{(i)}_n),x_i) = - \beta^\ast_{\id}(\gamma_n(\eta^{(c)}_n),z_c) + O(1) \xrightarrow{n\to \infty} -\infty.  \]
It now follows from \eqref{eq:expansion in the second coordinate} that
\[\sup_{x_1,x_2 \in \partial G} d^\ast_{\partial G}(\gamma_n(\xi^{(1)}_n)^{-1}( x_1) ,
 \gamma_n(\xi^{(2)}_n)^{-1} (x_2))
  \xrightarrow{n\to \infty} 0.
\]
As $\sup d_{\Sigma}(\xi^{(1)}_n, \xi^{(1)}_n)$ also tends to zero, it follows from uniformity in the estimates that (cp) holds with respect to any constant $a> 0$ (see Definition \ref{def:LEEA}).

\medskip

\noindent\textsc{Step 2: (lebca)}
With $\pi$ as in Proposition \ref{prop:convergence-to-the-boundary} and $\Delta_{\pi(\xi)}$ referring to the Dirac measure at $\pi(\xi)$, define $dm:= d\Delta_{\pi(\xi)}d\nu(\xi)$. As
\[
\gamma(\xi)^{-1}\pi(\xi) = \gamma(\xi)^{-1}  \lim_{n \to \infty} \gamma_n (\xi) = \lim_{n \to \infty} \gamma_n (\theta\xi ),\]
it follows that $m$ shares the properties of $\nu$, that is,  $m$ is $T$-invariant, exact and $\cL_T^\ast(m)=m$. Hence, by Lin's criterion (\cite{Lin--Mixing-For-Markov-Operators--ZFWUVG1971}), we have that $\lim_n \|\cL_T^n(b^\ast) - \int b^\ast dm\|_1 = 0$, for $b^\ast(\xi,x):= \beta^\ast_{\id}(\gamma(\xi),x)$.

On the other hand, it follows from Corollary \ref{cor:pos-lyapunov-exponent} that $\int b^\ast dm> 0$. Furthermore, as $\pi(U) = \partial G$ for any open set $U$, this implies that $\lim_{n \to \infty} \cL_T^n(b^\ast) = m(b^\ast)$ on a dense set. Hence, for $k$ to be specified later, and $S_k b^\ast := \sum_{i=0}^{k-1} b^\ast\circ T^i$,
\[\lim_{n \to \infty} \cL_T^n\left( \textstyle S_k b^\ast \right) = \sum_{i=0}^{k-1} \lim_{n \to \infty} \cL_T^n\left( b^\ast \right)  = km(b^\ast)  \]
on a dense set. However, even though $b^\ast$ is not necessarily continuous, it is known that there exists a uniform constant $C$, only depending on the hyperbolicity of $G$ such that, for any sequence $(x_n) \to x$ and $\gamma \in G$,  $\limsup |\beta^\ast_{\id}(\gamma,x_n) - \beta^\ast_{\id}(\gamma,x)| < C$ (see, e.g., Lemma 3.4.10 in \cite{Das-Simmons-Urbanski--Geometry-And-Dynamics-In--MSM2017}).
Hence, for any pair $(\gamma,x)$, there is $\epsilon > 0$ such that $|\beta^\ast_{\id}(\gamma,y) - \beta^\ast_{\id}(\gamma,x)|< 2C$ for all $d^\ast(x,y) < \epsilon$. For $n > k$ fixed, this now implies that for each pair $(\xi,x) \in \Sigma \times X$, there exists $\epsilon > 0 $  such that $d(\xi,\eta)< \epsilon$ and $d^\ast(x,y) < \epsilon$ imply that
 \begin{equation} \label{eq:uniform-distortion-for-busemann-in-average}
 \left|  \cL_T^n\left( S_k b^\ast \right)(\xi,x) -  \cL_T^n\left(S_k b^\ast \right)(\eta,y)  \right|  < 3C.
 \end{equation}
 By compactness, it now follows that there exists $\epsilon > 0$, independent of $(\xi,x)$ such that $d(\xi,\eta)< \epsilon$ and $d^\ast(x,y) < \epsilon$ imply that the left hand side of \eqref{eq:uniform-distortion-for-busemann-in-average} is smaller than $6C$.
This now allows us to choose $k$  such that
$\inf_{(\xi,x)} \cL_T^n(S_k b^\ast)((\xi,x)) > 0$ for all $n$ sufficiently large. Hence, there exists $n_0 \in k \N$ with
\begin{equation} \label{eq:lower-bound-for-busemann-in-average}
0 <   \inf_{(\xi,x)  \in \Sigma \times \partial G}   \sum_{\ell=0}^{n_0/k-1} \cL_T^{k \ell}(S_k b^\ast)(\xi,x) = \inf_{(\xi,x)  \in \Sigma \times \partial G}  \cL_T^{n_0}(S_{n_0} b^\ast) (\xi,x).
\end{equation}
Using $\cL_T^{n_0}(\mathbf{1}) = \mathbf{1}$ and $D := \|S_{n_0} b^\ast\|_\infty$, it follows for $t \geq  0$ that
\begin{align}\label{eq:lower-bound-for-square-of-busemann-in-average}
 \cL_T^{n_0} ( (S_{n_0} b^\ast)^2 e^{-t S_{n_0} b^\ast})
& \leq D^2 e^{tD}.
\end{align}
Now consider the analytic function
$ F(t) := \cL_T^{n_0}\left( \textstyle e^{-tS_{n_0}(b^\ast)}\right)$.
Then \eqref{eq:lower-bound-for-busemann-in-average} and \eqref{eq:lower-bound-for-square-of-busemann-in-average} imply that $F'(0) < 0$ and $F''(t) \leq D^2e^{tD}$. Therefore, there is $t_0>0$ such that
\[s(t) :=  \sup \{F(t)(\xi,x) :(\xi,x) \in \Sigma \times \partial G \} < 1\]
for all $0 < t \leq t_0$.
It then follows that $\cL_T^{m n_0}\left( \textstyle e^{-t S_{m n_0}(b^\ast)}\right)< s(t)^m$. It now remains to make the following choices for $t$ and $m$, as well as for the parameter $\lambda^\ast$ in the definition of the visual metric on $\partial G$ (see \eqref{eq:parameter_for_metric_on_Gromov_boundary}).

If $t_0 \geq - \log \lambda^\ast $, set $t:= - \log \lambda^\ast$. On the other hand, if $t_0 <  - \log \lambda^\ast $, set $t=t_0$ and  increase $\lambda^\ast$ to
$\lambda^\ast := \exp(-  t_0)$, which corresponds to
a change of metric from $d^\ast$ to $ (d^\ast)^\alpha$, for $\alpha:= - t_0 / \log \lambda^\ast $.

With respect to these choices, it now remains to choose $m$ so large that $s^m$ absorbs the multiplicative constant in
\eqref{eq:expansion in the second coordinate}
and $\exp (t6C)$ with $C$ as in \eqref{eq:uniform-distortion-for-busemann-in-average}. With respect to these parameters, it now remains to repeat the above argument based on the uniformity of $$\limsup |\beta^\ast_{\id}(\gamma,x_n) - \beta^\ast_{\id}(\gamma,x)| < C$$ and compactness in order to obtain  that $T$ is locally contracting in average with respect to $(d^\ast_{\partial G})^\alpha$.
\end{proof}

We are now in position to give a rough characterization of the measure $\mu$ given by Theorem \ref{theo:Ruelle}.
As the topology of $\partial G$ is invariant under quasi-isometry, we do not specify the metric in the following result.

\begin{proposition} \label{prop:dense-support}
Assume that (H$_1$ - H$_4$)  hold, and that $(\Sigma,\theta)$ is topologically mixing. Then $\dd \mu = \dd \Delta_{\pi(\xi)} \dd \nu(\xi)$, and $\mu(U)>0$ for any open set $U \subset \Sigma \times X$.
\end{proposition}

\begin{proof} Set $ \dd m:= \dd \Delta_{\pi(\xi)} \dd \nu(\xi)$. As can easily be verified, $\cL_T^\ast(m) = m$. By uniqueness of the $\cL_T^\ast$-invariant measure, $m = \mu$. So it remains to prove the second assertion. Assume that $U_1,U_2$ are sets of the form
\[
U_i = \{ x : \beta_{\id}(g_i ,x) <  - t_i \} \subset \partial G,
\]
for $g_1,g_2 \in G$ and  $t_1,t_2 > 0$ with $t_2 > d_w(\id,g_2)/2$. Now recall that, by Lemma 2.4 in \cite{Gouezel--Local-Limit-Theorem-For--JAMS2014}, there exists $C > 0$ such that, for any $g_1, g_2 \in G$, there exists $b \in G$ with $d_w(\id,b) \leq C$ and $ d_w(\id, g_1bg_2) \geq  d_w(\id, g_1) + d_w(\id, g_2)$.

\begin{figure}[ht]
\centering\def\svgwidth{0.6\textwidth}
\begingroup%
  \makeatletter%
  \providecommand\color[2][]{%
    \errmessage{(Inkscape) Color is used for the text in Inkscape, but the package 'color.sty' is not loaded}%
    \renewcommand\color[2][]{}%
  }%
  \providecommand\transparent[1]{%
    \errmessage{(Inkscape) Transparency is used (non-zero) for the text in Inkscape, but the package 'transparent.sty' is not loaded}%
    \renewcommand\transparent[1]{}%
  }%
  \providecommand\rotatebox[2]{#2}%
  \newcommand*\fsize{\dimexpr\f@size pt\relax}%
  \newcommand*\lineheight[1]{\fontsize{\fsize}{#1\fsize}\selectfont}%
  \ifx\svgwidth\undefined%
    \setlength{\unitlength}{502.1692279bp}%
    \ifx\svgscale\undefined%
      \relax%
    \else%
      \setlength{\unitlength}{\unitlength * \real{\svgscale}}%
    \fi%
  \else%
    \setlength{\unitlength}{\svgwidth}%
  \fi%
  \global\let\svgwidth\undefined%
  \global\let\svgscale\undefined%
  \makeatother%
  \begin{picture}(1,0.2440915)%
    \lineheight{1}%
    \setlength\tabcolsep{0pt}%
    \put(0,0){\includegraphics[width=\unitlength,page=1]{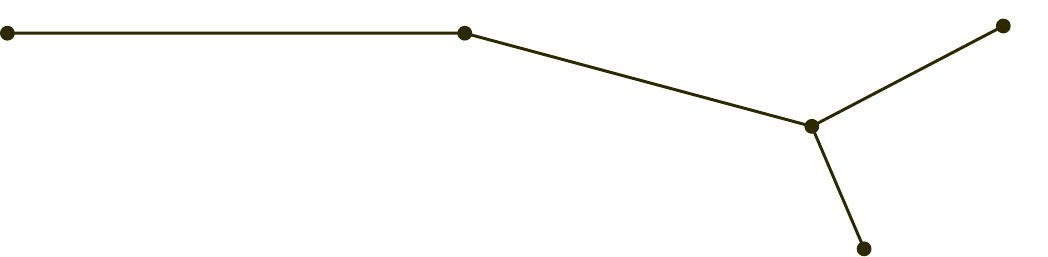}}%
    \put(0.8573589,0.00648475){\color[rgb]{0.17254902,0.16078431,0}\makebox(0,0)[lt]{\lineheight{0.72000158}\smash{\begin{tabular}[t]{l}$g_1 b g_2$\end{tabular}}}}%
    \put(0.42457283,0.22935084){\color[rgb]{0,0,0}\makebox(0,0)[lt]{\lineheight{0.72000158}\smash{\begin{tabular}[t]{l}$g_1b $\end{tabular}}}}%
    \put(0.00340013,0.22935084){\color[rgb]{0,0,0}\makebox(0,0)[lt]{\lineheight{0.72000158}\smash{\begin{tabular}[t]{l}$\id$\end{tabular}}}}%
    \put(0.9682141,0.22935084){\color[rgb]{0,0,0}\makebox(0,0)[lt]{\lineheight{0.72000158}\smash{\begin{tabular}[t]{l}$z$\end{tabular}}}}%
  \end{picture}%
\endgroup%

\caption{Tree approximation}
\label{fig:tree-approx}
\end{figure}
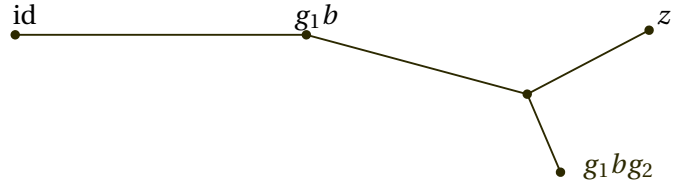

Now assume that $x \in \partial G$ satisfies $\beta_{g_1 b }(g_2 ,x) <  - t_2$, and that $z \in G$ is close to $x$ such that $\beta_{g_1 b }(g_2 ,x) = d_w(g_1 b g_2,z) -  d_w(g_1 b  ,z) + O(1)$. It then follows by a tree approximation
(cf. \cite[Th. 12]{Ghys-DelaHarpe--Le-Bord-Dun-Espace--1990_2} or \cite[Th. 2.1]{Gouezel--Local-Limit-Theorem-For--JAMS2014})  of the four points $\id$, $g_1 b$, $g_1 b g_2$, and $z$ (cf. Figure \ref{fig:tree-approx}) that
\begin{align*}
d_w(g_1,z) -  d_w(\id,z) = d_w(g_1,z) - d_w(\id,g_1 b) - d_w(g_1 b ,z) + O(1)
= - d_w(\id,g_1) + O(1) .
\end{align*}
By letting $z$ tend to $x$, it follows that there exists a universal constant $C$ such that $\beta_{\id}(g_1,x) > - d_w(\id,g_1) +C$. Hence, for $t_1 < d_w(\id,g_1)  - C$, it follows that $(g_1b)(U_1) \supset U_2$.

Now suppose that $w_1,w_2$ are finite words. It now follows from topological transitivity of $T$ that there is a further finite word $v$ such that $w_1vw_2$ is admissible and  $\gamma_{|w_1v|}(\xi) = g_1b$, for any $\xi \in [w_1vw_2]$. In particular,
$T^{|w_1v|} ([w_1] \times U_1) \supset [w_2]\times U_2$.
So it remains to observe that sets of the $[w_i] \times U_i$ form a bases of the topology and to choose $[w_2]\times U_2$ with $\mu([w_2]\times U_2) > 0$. Then
$\mu\left([w_1]\times U_1\right) > \exp(|w_1v| \min \varphi)  \mu\left([w_2]\times U_2\right) > 0$, provided that $d_w(\id,g_1) > C$. This proves the second assertion as $[w_1] \times U_1$ can be chosen arbitrarily.
\end{proof}

\subsection{Fluctuations of the drift with respect to the Green metric}
In order to obtain a precise description of the fluctuations of the Busemann function by applying Theorem \ref{theo:vasip}, one needs to verify that $x \mapsto \beta_{\id}(g,x)$ is Hölder continuous. As this does not necessarily hold with respect to the word metric $d_\mathrm{w}$, we now construct the following dynamically defined metric on $G$. Fix $\xi \in \Sigma$ and, for $D \geq  0$ to be specified below and $g,h \in G$, define
\[ d_\mathrm{pre}(g,h) := -\log \left( \G_1(\X_g)(\xi,h) \G_1(\X_h)(\xi,g)  \right) + D. \]
As there is $C\geq 1$ such that $\G_1(\X_g)(\xi,j) \G_1(\X_j)(\xi,h) \leq C \G_1(\X_g)(\xi,h)$, it follows that  $d_\mathrm{pre}$ satisfies the triangle inequality for any $D$ with $D \geq 2 \log C$.  Furthermore, it follows from the definition of $\G$ and the fact that $\gamma(\Sigma)$ is finite that there is $a > 0$ such that
$- \log \G_1(\X_g)(\xi,h) <  a d_\mathrm{w}(g,h)$.  So it remains to choose $D \geq 2 \log C$ so large that $ d_\mathrm{pre}(g,h) > 0$ for all $g,h \in G$. It is now straightforward to check that
\[
d_\G(g,h) :=  \min \left\{ 2 a d_\mathrm{w}(g,h), d_\mathrm{pre}(g,h) \right\}
\]
is a metric, and that this metric coincides with $d_\mathrm{pre}(g,h)$, provided that $a d_\mathrm{w}(g,h) > D$. Due to the obvious analogy to the {Green metric} for random walks on groups (\cite{Blachere-Brofferio--Internal-Diffusion-Limited-Aggregation--PTR2007,Blachere-Haissinsky-Mathieu--Asymptotic-Entropy-And-Green--AP2008,Bjorklund--Central-Limit-Theorems-For--JTP-2010}), we also refer to $d_\G$ as the \emph{Green metric}.

We now have to assume that (H$_1$ - H$_4$) hold. By Theorem 5.2 (v) in \cite{Bispo-Stadlbauer--The-Martin-Boundary-Of--IJM2023} there is  $  b  > 0$ and $K> 0$ such that $ d_\mathrm{pre}(g,h)   <  b d_\mathrm{w}(g,h)$ provided that $d_\mathrm{w}(g,h)> K$. Hence, $d_\mathrm{w}$ and $ d_\G$  are roughly isometric. This implies that $(G,d_\G)$ is hyperbolic and that Proposition \ref{prop:convergence-to-the-boundary} and Theorem \ref{theo:positive-drift} hold with respect to the new metric. Furthermore, as
$(G,d_\G)$ is hyperbolic, we may replace
 $d_\mathrm{w}$ with $d_\G$
in the definitions of the Gromov product, the visual metric and the Busemann function. In order to distinguish these objects, we will write $(g \cdot h)_\mathbf{o}^\ast$, $d_{\partial G}^\ast$ and  $\beta^{\ast}$ for these new objects.
It then follows as above that Corollary \ref{cor:pos-lyapunov-exponent} also holds with respect to  $\beta^{\ast}$.
The main advantage, however, is the following. As the Gouëzel–Lalley inequality holds for $T$ (see \cite[Th. 4.6 (ii)]{Bispo-Stadlbauer--The-Martin-Boundary-Of--IJM2023}), there are $C > 0$ and $\lambda_\ast \in [0,1)$ such that
\begin{equation}
\label{eq:busemann is hoelder}
\left|\beta^\ast_{\id}(g, x) - \beta^\ast_{\id}(g, y) \right| \leq C \lambda_\ast^{|(x \cdot y)_{\id} - (g \cdot x)_{\id}|},
\end{equation}
for $g \in G$ and $x,y \in \partial G$ in a configuration as depicted in Figure \ref{fig:GL-inequality}.
It hence follows from rough isometry and the definition of $d^\ast_{\partial G}$ that $x \mapsto \beta^\ast_{\id}(g, x)$ is Hölder continuous.

\begin{figure}[ht]
\centering\def\svgwidth{0.7\textwidth}
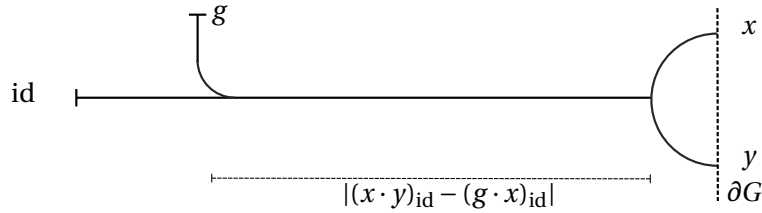
\caption{Gouëzel–Lalley inequality }
\label{fig:GL-inequality}
\end{figure}

As a consequence of Theorem \ref{theo:vasip}, Theorem \ref{theo:hyperbolic-groups-and-lebca} and the dichotomy in Proposition \ref{prop:dichotomy}, it now follows that $\beta_{\id}^\ast( \gamma_n(\xi), \pi(\xi))$ satisfies a non-degenerate almost sure invariance principle. Hence, by applying Proposition \ref{prop:convergence-to-the-boundary}, the fluctuations of $d_\G(\id, \gamma_n)$ with respect to the Green metric are now given in terms of an almost sure invariance principle.

\begin{corollary} \label{cor:fluctuation-of-drift}
  Assume that (H$_1$ - H$_4$)  hold, and that $(\Sigma,\theta)$ is topologically mixing. Then, for any $\lambda>1/4$ and after eventually extending the probability space $(\Sigma, \mu)$, there exist $\mathfrak{m} > 0$, $\sigma > 0$ and a standard Brownian motion $(B_t)$ such that, for $\mu$-a.e. $\xi$,
  \[
  d_\G(\id, \gamma_n(\xi)) = \mathfrak{m} n + \sigma B_n + o(n^\lambda).
  \]
\end{corollary}

\begin{proof}
 It only remains to show that  $\sigma > 0$. By Proposition \ref{prop:dense-support} and the remark thereafter,  it is sufficient to construct two periodic fixed points of the same period with different values of $\beta_{\id}$. However, by topological transitivity, there exists $\xi_1,\xi_2 \in \Sigma$, $n_1,n_2 \in \N$ with
 $\theta^{n_i}(\xi_i) = \xi_i$, $\gamma_{n_1}(\xi_1) = \id$ and $\gamma_{n_2}(\xi_2) \neq \id$.
 Set $x = \lim_{k \to \infty} (\gamma_{n_2}(\xi_2))^k \in \partial G$.
 Then $T^{n_i}(\xi_i,x)=(\xi_i,x)$ for $i=1,2$.
 Moreover, $(\xi_i,x)$ are periodic points of period $k n_1 n_2$ and $S_{k n_1 n_2} \beta_{\id} (\xi_1,x) =  O(1)$ and
 $S_{kn_1n_2} \beta_{\id} (\xi_2,x) \asymp k$.
\end{proof}

\section{An ASIP for the geodesic flow on a regular cover}
\label{sec:regular cover}

The aim of this section is to discuss a further application of the above theory as well as to provide a non-trivial application of the vector-valued version of the ASIP.

In order to do so, recall that a \emph{CAT($-1$) space} $X$ is a proper geodesic space such that each geodesic triangle is thinner than a comparison triangle in the hyperbolic plane with constant curvature -1. For example, each geodesic space of pinched negative curvature has this property.
As a consequence of CAT(-1), $X$ is strongly hyperbolic and therefore comes with its own Gromov product, a visual boundary $\partial X$ and a Busemann function $\beta_{X}$. Furthermore, after fixing a point $\mathbf{o}\in X$, the space of parameterized, oriented geodesics can be identified with
\[  \mathcal{G} := \{ (\xi,\eta,s) \in \partial X \times \partial X \times \R: \xi \neq \eta \}\]
by mapping a geodesic $\gamma : \R \to X$ to the triple $( \gamma(\infty),   \gamma(-\infty),s)$, where $\gamma(\pm \infty) = \lim_{t \to \pm \infty} \gamma(t)$ and $s$ is determined by $d(\mathbf{o}, \gamma(-s)) = \min_t d(\mathbf{o}, \gamma(t))$. In this representation, the geodesic flow $(g_t)$ is defined by $g_t(\xi,\eta,s) = (\xi,\eta,s+ t)$.

Now assume that $\Gamma$ is a discrete subgroup of the isometry group $\hbox{Isom}(X)$ such that $\Gamma$ is acting freely and properly discontinuously on $X$. In particular, there is a canonical metric such that the covering map $X \to X/\Gamma$ is a local isometry. Furthermore, by extending the action of $\Gamma$ to $\partial X$,
one then obtains the geodesic flow acting on $\mathcal{G}/\Gamma$.

In order to introduce the relevant measures, recall that the abscissa of convergence of $\Gamma$ is defined by
 \[ \delta_\Gamma := \sup\left\{s \geq 0: \sum_{\gamma \in \Gamma} e^{-s d(\mathbf{o},\gamma(\mathbf{o}))} = \infty  \right\},
 \]
and that  $\Gamma$ is said to be of divergence type if $\sum_{g \in \Gamma} e^{-\delta_\Gamma d(\mathbf{o},g(\mathbf{o}))} = \infty$. Moreover, a Borel probability measure $\mu$ on $\partial X$ is referred to as an $s$-conformal measure,  for $s > 0$, if
 \begin{equation}\label{eq:conformal measure}
 \frac{\dd \mu \circ h}{\dd \mu}(\xi) =   e^{- s \beta_{X,\mathbf{o}}(h^{-1}(\mathbf{o}),\xi) }   \end{equation}
 for any Borel subset $A$ of $\partial X$ and any $h \in \Gamma$.  The limit set of $\Gamma$ is defined by $\Lambda(\Gamma) := \overline{\Gamma(\mathbf{o})} \cap \partial X$.

 The Hopf--Tsuji theorem states that $\Gamma$ is of divergence type if and only if, for any $\delta_\Gamma$-conformal measure $\mu$, the action of $\Gamma$ on $(\partial X,\mu)^2$ is ergodic.  Moreover, for groups of divergence type, the $\delta$-conformal measure is unique (for the setting of CAT($-1$) spaces and the definition of $\delta$-conformality, see \cite{Das-Simmons-Urbanski--Geometry-And-Dynamics-In--MSM2017}).
 Now let  $\mathcal{G}_\Gamma$ refer to those elements in $\mathcal{G}$ with endpoints in $\Lambda(\Gamma)$, and to $p: \mathcal{G}_\Gamma \to X$ to the canonical projection. We then say that  $\Gamma$ is \emph{convex-cocompact} if $\mathcal{G}_\Gamma/\Gamma$ is compact. In this case, it is well known that   $\delta_\Gamma$ is finite and that the group is of divergence type.
  From the viewpoint of the geodesic flow, being of divergence type is equivalent to conservativity and ergodicity of the geodesic flow with respect to the flow-invariant measure
 \begin{equation}\label{eq:flow invariant measure}
 \dd m  = e^{2 \delta_\Gamma ( \xi \cdot \eta)_{X, \mathbf{o}} } \dd \mu(\xi) \dd \mu(\eta) \dd s,
 \end{equation}
where $\mu$ refers to the  $\delta_\Gamma$-conformal measure.

We now recall the notion of a Galois (or regular) cover $\mathcal{G}/\Gamma$. That is, if $N$ is a normal subgroup of $\Gamma$, then the cover $\pi: \mathcal{G}/N \to \mathcal{G}/\Gamma$ is referred to as regular cover. Moreover, these covers are also referred to as Galois covers in order to put emphasis on the fact that each fiber of $\pi$ carries the structure of the group $G/N$. In order to state the next theorem, observe that $m$ in \eqref{eq:flow invariant measure} induces  flow invariant measures $m_\Gamma$ and $m_N$ on $\mathcal{G}/\Gamma$ and $\mathcal{G}/N$, respectively.
Note that $X/N$ comes with a natural metric
\[ d_{X/N}(x,y) := \inf\left\{ \ell(c) : c \hbox{ is a curve from } x \hbox{ to } y \hbox{ in }X/N \right\} \]
and that $G :=\Gamma/N$ can be identified with a group of deck transformations by normality. As $G$ acts cocompactly on $p(\mathcal{G}_\Gamma)/N \subset X/N$, it follows from the Švarc–Milnor theorem, that the Cayley graph of $G$ with respect to the word metric embeds quasi-isometrically into
$p(\mathcal{G}_\Gamma)/N$. Furthermore, as it easily can be seen, for any pair $x,y \in p(\mathcal{G}_\Gamma)/N$ there is $g \in G$ such that $d_{X/N}(gx,y)$ is smaller than the diameter of $p(\mathcal{G}_\Gamma)/\Gamma \subset X/\Gamma$. In particular, it follows that $(p(\mathcal{G}_\Gamma)/N, d_{X/N})$ is a Gromov hyperbolic space which is not necessarily strongly hyperbolic. On the other hand, the embedding of the Cayley graph provides us with an identification of $\partial G$ with the Gromov boundary of $p(\mathcal{G}_\Gamma)/N$.

In order to circumvent the problem that the Busemann function might not be Hölder continuous, one has to pass as in Corollary \ref{cor:fluctuation-of-drift} to the Green metric. However, as the statement of the theorem below does not require a metric, it suffices to consider
\[
d_{X/N}^\ast(x,y) := - \log \sum_{\gamma \; \mathrm{is}\; \mathrm{geod.} \; \mathrm{arc} \; x \to y } e^{- \delta_\Gamma \ell(\gamma)}
=  - \log \sum_{g \in N} e^{- \delta_\Gamma d_X(x^\ast,gy^\ast)},
\]
for $x,y \in X/N$ and some $x^\ast,y^\ast \in X$ which project under $X \to X/N$ to $x$ and $y$, respectively. In particular, $\exp( - d_{X/N}^\ast)$ corresponds to the Poincaré series of $N$  at $\delta_\Gamma$.
Moreover, even though $d_{X/N}^\ast$ is not necessarily a metric, $d_{X/N}^\ast$ coincides with the Green metric up to a uniform error.

\begin{theorem}\label{theo:geometric-application} Suppose that $X$ is a CAT($-1$) space, that $\Gamma$ is a  convex-cocompact, discrete subgroup of the group of orientation-preserving isometries of $X$.  Furthermore, assume that $N$ is a normal subgroup of $\Gamma$ such that $\Gamma/N$ is a non-elementary, word-hyperbolic group and such that the geodesic flow on $\mathcal{G}(X/N)$ is topologically transitive. Then the following holds
\begin{enumerate}
 \item There exist $c_1,c_2 > 0$ such that
 $ c_1 d_Y  - c_2 \leq  d_{X/N}^\ast \leq \delta_\Gamma d_Y$ on  $p(\mathcal{G}_\Gamma)/N$.
  \item $d_{X/N}(p(g_s (\xi)),\mathfrak{g}_\xi([0,\infty) ) )  = O(\log s)$ as $s \to \infty$,  for any geodesic half ray $\mathfrak{g}_\xi$ from  $\mathbf{o}$ to $\pi(\xi)$,
 \item After extending the probability space $(\mathcal{G}/\Gamma, (m_\Gamma(\mathcal{G}/\Gamma))^{-1} m_\Gamma )$, where $m$ is as in \eqref{eq:flow invariant measure}, there are $\mathfrak{m},\sigma > 0$ and a standard Brownian motion $(B_t)$ such that, for any $\epsilon > 0$, and $m$-a.e. $\zeta \in \mathcal{G}/N$
 \[
d_{X/N}^\ast( p(\zeta), p(g_t(\zeta)) = \mathfrak{m} t + \sigma B_t + o(t^{1/4 + \epsilon}) \hbox{ as $t \to \infty$}.
 \]
 \item If $p(\mathcal{G}_\Gamma/N)$ is strongly hyperbolic, then one may replace $d_Y^\ast$ with $d_Y$ in the above approximation by a Brownian motion.
\end{enumerate}

\end{theorem}

\begin{proof} The Markov coding of the geodesic flow on $\mathcal{G}/\Gamma$ in \cite{Constantine-Lafont-Thompson--Strong-Symbolic-Dynamics-For--JLP-M2020} provides us with a two-sided subshift of finite type $(\Sigma_\ast,\theta)$, a map $\kappa : \Sigma_\ast \to  \Gamma$, which depends only on the zeroth coordinate, and a bounded and Hölder continuous roof function such that the flow on $\mathcal{G}/\Gamma$ is isomorphic to the suspension flow over the two-sided shift under this roof function with respect to $m$.

As shown in Step 1 in the proof of Theorem 6.2 in \cite{Bispo-Stadlbauer--The-Martin-Boundary-Of--IJM2023}, the geodesic flow on $\mathcal{G}/N$ can be identified with a suspension flow under the same roof function over the skew product $(\xi,[g]) \mapsto (\theta(\xi), [g\gamma(\xi)])$ on $\Sigma_\ast \times \Gamma/N$.
Due to the Markov property, the first return map has a maximal non-invertible factor of the form \eqref{eq:random-walk-with-stationary-increments}. Moreover, as the roof function is Hölder continuous, there is a function $h : \Sigma \to [0,\infty)$ which is cohomologous to the roof function of the two-sided subshift. Hence, in order to study the fluctuations of $d( \zeta, g_t(\zeta))$ and $d_\G( \zeta, g_t(\zeta))$ for up to a bounded error, it is sufficient to study the suspension semi-flow over $S$ with Lipschitz continuous height function $h$.

On the level of measures, note that $m$, as defined in \eqref{eq:flow invariant measure}, projects to an invariant measure on the two-sided subshift and then to a $\theta$-invariant measure $\nu$ on $\Sigma$, which is absolutely continuous with respect to the $\delta_\Gamma$-conformal measure $\mu$. As a consequence of \eqref{eq:conformal measure}, Ruelle's operator theorem and the fact that the Busemann function is Hölder continuous in CAT($-1$) spaces, it follows that $\dd \nu / \dd \mu$ is Hölder continuous and that $\varphi = \dd \nu / \dd \nu \circ \theta$ is a normalized,  Hölder continuous potential, cohomologous to   $ \delta_\Gamma \beta_\mathbf{o}(\gamma(\xi),\pi_+(\xi))$, with $\pi_+(\xi)\in \partial X$ referring to the endpoint of the geodesic half-ray coded by $\xi$.

By combining Proposition \ref{prop:convergence-to-the-boundary}, which states that  $\gamma_n(\xi)$ is at most $O(\log n)$ away from the half-ray from $\id$ to $\pi(\xi)$, with the observation that there is a uniform bound for the distance of elements in $p(\mathcal{G}_\Gamma/N)$ to the embedded Cayley-graph of $\Gamma/N$, it follows that
\[d_{X/N}(\gamma, p(g^\ast_t(\xi))) = O(\log t)\]
almost surely, where $\gamma$ is some geodesic half ray from $\mathbf{o}$ to $\pi(\xi)$ and $(g^\ast_t)$ refers to the suspension semi-flow. In particular, the second assertion holds.

After changing the metric to the Green metric, depending on whether the Busemann function is Hölder continuous, it hence remains to apply Theorem \ref{theo:vasip-flow} in analogy to  Corollary \ref{cor:fluctuation-of-drift} in order to obtain the almost sure invariance principle for the drift with respect to the Green metric on $G$, but this time with respect to the semi-flow. The nontriviality of $\sigma$ follows from an  argument similar to that in Corollary \ref{cor:fluctuation-of-drift}. This proves the third and fourth assertion.

It remains to relate the Green metric with $d^\ast_{X/N}$. In order to do so, we recall from \cite{Bispo-Stadlbauer--The-Martin-Boundary-Of--IJM2023} that the symmetry condition (S) holds, which implies that, for $g,h \in \Gamma$ and with $[g],[h]$ referring to the corresponding elements in $\Gamma/N$,
\[
d_\G ([g],[h])  = - 2 \log \G(\X_{[g]}),(\xi, [h]) + O(1).
\]
However, $\G(\X_{[g]}),(\xi, [h])$ is a sum over terms of the form $\exp(\varphi_n(\xi^ast))$, for some $n\in \N$ and $\xi^\ast$ with $T^n(\xi_\ast,[g]) = (\xi, [h])$. Moreover, as $\Gamma$ is convex-cocompact and $\varphi_n(\xi) = \delta_\Gamma \beta_\mathbf{o}(\gamma_n(\xi),\xi) + O(1)$, it follows that
\[
d_\G ([g],[h])  = - 2 \log \sum_{z \in  g^{-1}h N }  e^{- \delta_\Gamma d(\mathbf{o} , z \mathbf{o} )}  + O(1).
\]
Hence, $d_\G = d^\ast_{X/N} + O(1)$. The first assertion now easily follows from the fact that $d_\G$ is quasi-isometric to $d_w$.
\end{proof}

\subsection*{Acknowledgments}

The authors gratefully acknowledge the financial support provided by the Coordenação de Aperfeiçoamento de Pessoal de Nível Superior – Brasil (CAPES), through the Programa de Excelência Acadêmica (PROEX), Finance Code 001.

The first author acknowledges the Fundação Carlos Chagas Filho de Amparo à Pesquisa do Estado do Rio de Janeiro (FAPERJ) for the support provided through the Nota 10 program (Process E-26/201.177/2020).

The second author acknowledges the Conselho Nacional de Desenvolvimento Científico e Tecnológico (CNPq) for the Productivity in Research fellowship (Grant 314896/2023-6) and the Fundação Carlos Chagas Filho de Amparo à Pesquisa do Estado do Rio de Janeiro (FAPERJ) for the Cientista do Nosso Estado (CNE) fellowship (Process E-26/204.324/2024).

\setlength{\emergencystretch}{2em}

\printbibliography

\end{document}